\documentclass{article}

\usepackage{amsmath}
\usepackage{amssymb}
\usepackage{amsfonts}
\usepackage{amsthm}
\usepackage{mathtools}
\usepackage{cleveref}
\usepackage{fullpage}
\usepackage{xcolor}
\usepackage{graphicx}
\usepackage{verbatim}

\usepackage{colonequals}

\newcommand{\R}{\mathbb{R}}
\newcommand{\N}{\mathbb{N}}
\newcommand{\argmin}{{\mathrm{argmin}}}
\renewcommand{\d}{\mathrm{d}}

\newcommand{\mfd}{{\mathcal M}}
\newcommand{\partpt}[1]{\pi_{#1}}

\newcommand{\lift}[2]{#1^{\#}T#2}
\newcommand{\Rm}{{\mathrm{R}}} 

\newcommand{\linInt}{{\mathcal I^1_0}}
\newcommand{\constInt}{{\mathcal I^0_0}}

\newtheorem{theorem}{Theorem}
\newtheorem{proposition}[theorem]{Proposition}
\newtheorem{lemma}[theorem]{Lemma}
\newtheorem{corollary}[theorem]{Corollary}
\newtheorem{definition}[theorem]{Definition}

\newtheorem{remark}[theorem]{Remark}

\newcommand{\notinclude}[1]{}

\graphicspath{{./figures/},{.}}

\title{Quartic $L^p$-convergence of cubic Riemannian splines}
\author{Hanne Hardering \and Benedikt Wirth}
\date{}

\begin{document}

\maketitle

\begin{abstract}
We prove quartic convergence of cubic spline interpolation for curves into Riemannian manifolds as the grid size of the interpolation grid tends to zero.
In contrast to cubic spline interpolation in Euclidean space, where this result is classical, the interpolation operator is no longer linear.
Still, concepts from the linear setting may be generalized to the Riemannian case,
where we try to use intrinsic Riemannian formulations and avoid charts as much as possible.
\end{abstract}

\section{Introduction}
Let a smooth function $\gamma:[0,1]\to\R$ be given as well as a grid on $[0,1]$ of knots $t_i=ih$, $i=0,\ldots,N$,
where $h=\frac1N$ denotes the knot distance or grid size.
A standard approximation of $\gamma$ is obtained by evaluating $\gamma$ at the knots
and then computing the cubic spline interpolant $\gamma_h:[0,1]\to\R$ of the resulting data $(t_i,\gamma(t_i))$, $i=0,\ldots,N$,
for instance with Hermite boundary conditions $\dot\gamma_h(0)=\dot\gamma(0)$, $\dot\gamma_h(1)=\dot\gamma(1)$.
It is a classic result that this approximation is of quartic order, $\|\gamma_h-\gamma\|_{L^\infty}\leq ch^4$,
proved for instance by de Boor \cite{dB74} (analogous convergence results hold for spline interpolation of different order).

The well-known minimum curvature property of cubic splines states
that cubic spline interpolants uniquely minimize the accumulated squared acceleration $\int_0^1|\ddot\gamma_h(t)|^2\,\d t$ among all functions interpolating $(t_i,\gamma(t_i))$, $i=0,\ldots,N$,
which allows to generalize the notion of a cubic spline interpolant to functions or curves $\gamma:[0,1]\to\mfd$ with values in a smooth Riemannian manifold $\mfd$.

\begin{definition}[Riemannian cubic spline interpolation]\label{def:cubicSpline}
Let $\gamma:[0,1]\to\mfd$ be a continuous curve and set $t_i=ih$ for $i=0,\ldots,N$ and $h=\frac1N$.
A corresponding \emph{cubic spline interpolation} of $\gamma$ is defined as
\begin{equation*}
\gamma_h\in\argmin\left\{\int_0^1|D_t^2\gamma_h|^2\,\d t\,\middle|\,\gamma_h(t_i)=\gamma(t_i)\text{ for }i=0,\ldots,N\text{ and }\dot\gamma_h(t)=\dot\gamma(t)\text{ for }t=0,1\right\}.
\end{equation*}
\end{definition}

Above, $D_t^2\gamma_h$ denotes the intrinsic second derivative, the covariant derivative of $\dot\gamma_h$ along the curve $\gamma_h$,
and $|\cdot|$ represents the Riemannian norm on the tangent space bundle $T\mfd$.
A natural question is whether the approximation properties of cubic splines transfer from the setting of real-valued functions to the manifold-valued generalization.
Our main result is that they do.

\begin{theorem}[Quartic convergence of Riemannian cubic spline interpolation]\label{thm:mainResultCubic}
Let $\mfd$ be complete with bounded Riemann curvature and its derivatives.
Let $\gamma:[0,1]\to\mfd$ be four times differentiable.
For $h$ small enough (depending on $\mfd$ and $\gamma$)
\begin{enumerate}
\item\label{enm:wellPosednessCubic}
the cubic spline interpolation $\gamma_h$ exists and is unique,
\item
and it satisfies
\begin{equation*}
\|d(\gamma_h,\gamma)\|_{L^p}
\leq ch^4,
\end{equation*}
where the constant $c>0$ only depends on $\mfd$, $\gamma$, and $p\in[1,\infty]$.
\end{enumerate}
\end{theorem}

Above, $d(\cdot,\cdot)$ denotes the Riemannian distance and $\|d(\gamma_h,\gamma)\|_{L^p}$ the $L^p$-norm of the function $t\mapsto d(\gamma_h(t),\gamma(t))$ for $t\in[0,1]$.
Note that the conditions for existence and uniqueness can be a little relaxed, see \cref{thm:wellPosedness}.
One could envision two strategies to arrive at this approximation result.
One strategy would be to reduce the problem to the real-valued case
by reformulating the Riemannian cubic spline interpolation as an interpolation problem in a linear space (using charts or a tangent space to the manifold)
which can be viewed as a perturbation of Euclidean or real-valued cubic spline interpolation.
The complementary strategy consists in the generalization of all ingredients to the Euclidean analysis to the Riemannian setting,
using a mostly coordinate-free intrinsic Riemannian formulation.
We will follow the latter strategy (but also discuss the former along the way).
In particular, the Euclidean proof in \cite{dB74} makes use of the fact
that the second derivative of the cubic spline interpolant $\gamma_h$ is the $L^2$-best approximation to the the second derivative of $\gamma$,
and for our proof in the nonlinear manifold setting we will quantify the deviation from a Riemannian generalization of that property.
This will result in an $L^\infty$-bound for the second derivative $D_t^2\gamma_h$, which is crucial for the convergence estimate.

For the sake of completeness, along the way we will deal with linear spline interpolation
since it is a widespread alternative for applications in Riemannian manifolds
and since its convergence properties follow without additional effort.

\begin{definition}[Riemannian linear spline interpolation]\label{def:linearSpline}
Let $\gamma:[0,1]\to\mfd$ be a continuous curve and set $t_i=ih$ for $i=0,\ldots,N$ and $h=\frac1N$.
A corresponding \emph{linear spline interpolation} of $\gamma$ is defined as
\begin{equation*}
\gamma_h\in\argmin\left\{\int_0^1|\dot\gamma_h|^2\,\d t\,\middle|\,\gamma_h(t_i)=\gamma(t_i)\text{ for }i=0,\ldots,N\right\}.
\end{equation*}
\end{definition}

\begin{theorem}[Quadratic convergence of Riemannian linear spline interpolation]\label{thm:mainResultLinear}
Let $\mfd$ be complete with bounded Riemann curvature and its derivatives.
Let $\gamma:[0,1]\to\mfd$ be twice differentiable.
For $h$ small enough (depending on $\mfd$ and $\gamma$)
\begin{enumerate}
\item\label{enm:wellPosednessLinear}
the linear spline interpolation $\gamma_h$ exists and is unique,
\item
and it satisfies
\begin{equation*}
\|d(\gamma_h,\gamma)\|_{L^p}
\leq ch^2,
\end{equation*}
where the constant $c>0$ only depends on $\mfd$, $\gamma$, and $p\in[1,\infty]$.
\end{enumerate}
\end{theorem}

Higher order spline interpolation of odd degree could in principle be tackled in a similar way as cubic spline interpolation, but of course requires more technical estimates.

In \cref{sec:notation} we will summarize the required basic Riemannian notions and our notation
after which we prove the well-posedness of (linear and) cubic spline interpolation in \cref{sec:wellPosedness} and derive the corresponding optimality conditions.
We then present two convergence proofs for the classical Euclidean setting in \cref{sec:Euclidean}
before adapting these to the Riemannian setting in \cref{sec:convergence}.

\section{Required Riemannian notions and preliminaries}\label{sec:notation}
Throughout we will assume $\mfd$ to be a smooth, complete, connected, finite-dimensional Riemannian manifold with Riemannian metric $g$
such that the Riemann curvature tensor and all its derivatives are bounded.
The tangent space at a point $p\in\mfd$ will be denoted $T_p\mfd$, and for $v,w\in T_p\mfd$ we will write
\begin{equation*}
(v,w)=g_p(v,w)
\qquad\text{as well as}\qquad
|v|=\sqrt{(v,v)},
\end{equation*}
where the explicit dependence on the base point $p$ is suppressed for better readability.
The induced Riemannian distance is denoted by $d:\mfd\times\mfd\to[0,\infty)$,
the Levi-Civita connection on $\mfd$ is denoted by $\nabla$,
and the Riemann curvature tensor is denoted by $\Rm$.
For an absolutely continuous curve $\gamma:[0,1]\to\mfd$ we will denote the parameterization variable by $t$
and its velocity by $\partial_t\gamma(t)=\dot\gamma(t)\in T_{\gamma(t)}\mfd$
(time derivatives of functions into a vector space will also be denoted with powers of $\partial_t$ or dots).
A lift of the curve $\gamma$ is a mapping $v$ from $[0,1]$ into the tangent bundle $T\mfd$ such that $v(t)\in T_{\gamma(t)}\mfd$ for all $t\in[0,1]$;
the vector space of all liftings is abbreviated $\lift\gamma\mfd$.
For a vector field $v\in\lift\gamma\mfd$ along the curve $\gamma$ and $k\in\N$ we will write
\begin{equation*}
D_t^kv=\nabla_{\dot\gamma}^kv,
\end{equation*}
where with a slight misuse of notation the covariant derivative of $v$ along $\gamma$ shall be
\begin{equation*}
\nabla_{\dot\gamma}v(t)=\lim_{\epsilon\to0}\tfrac1\epsilon(\partpt{\gamma}v(t+\epsilon)-v(t))
\end{equation*}
with $\partpt\gamma$ being the parallel transport along $\gamma$ from $\gamma(t+\epsilon)$ to $\gamma(t)$
(start and end point of the parallel transport along a curve will usually be clear from the context).
We will also write
\begin{equation*}
D_t^k\gamma=\nabla_{\dot\gamma}^{k-1}\dot\gamma
\end{equation*}
for $k\in\N$.
A family of curves $\gamma(s,\cdot):[0,1]\to\mfd$ will be parameterized by the variable $s\in\R$,
and analogously to before, for a family of vector fields $v(s,\cdot)$ along the curves $\gamma(s,\cdot)$ we introduce the notation
\begin{equation*}
D_sv(s,t)=\nabla_{\partial_s\gamma}v(s,t)=\lim_{\epsilon\to0}\tfrac1\epsilon(\pi_{\gamma(\cdot,t)}v(s+\epsilon,t)-v(s,t)),
\end{equation*}
where $\pi_{\gamma(\cdot,t)}$ is parallel transport along the curve $s\mapsto\gamma(s,t)$.
Below we collect a few classical Riemannian calculus rules to be used in the sequel.

\begin{lemma}[Differentiation rules]\label{thm:rules}
\begin{enumerate}
\item\label{enm:transportedDerivative}
Parallel transport along a curve commutes with differentiation along the curve in the following sense.
Let $\gamma:[0,1]\to\mfd$ be differentiable and denote by $\pi_\gamma:T_{\gamma(t)}\mfd\to T_{\gamma(0)}\mfd$ the parallel transport along $\gamma$,
then for any differentiable $v\in\lift\gamma\mfd$ we have
\begin{equation*}
\tfrac\d{\d t}\pi_\gamma v=\pi_\gamma D_tv.
\end{equation*}
\item
Let $\gamma:\R\times[0,1]\to\mfd$ be a differentiable family of differentiable curves, then
$D_s\partial_t\gamma=D_t\partial_s\gamma$.
\item
Let $\gamma:\R\times[0,1]\to\mfd$ be a differentiable family of differentiable curves and $v\in\lift\gamma\mfd$ differentiable, then
$D_sD_tv=D_tD_sv+\Rm(\partial_s\gamma,\partial_t\gamma)v$.
\end{enumerate}
\end{lemma}
\begin{proof}
\begin{enumerate}
\item
For all $w\in T_{\gamma(0)}\mfd$ we have
\begin{equation*}
(\tfrac\d{\d t}\pi_\gamma v,w)
=\tfrac\d{\d t}(\pi_\gamma v,w)
=\tfrac\d{\d t}(v,\pi_\gamma^{-1}w)
=(D_tv,\pi_\gamma^{-1}w)+(v,D_t\pi_\gamma^{-1}w)
=(D_tv,\pi_\gamma^{-1}w)
=(\pi_\gamma D_tv,w).
\end{equation*}
\item
Since the Levi-Civita connection is torsion-free, we have $D_s\partial_t\gamma=D_t\partial_s\gamma+[\partial_s\gamma,\partial_t\gamma]$ for the Lie bracket $[\cdot,\cdot]$ of vector fields.
Now for any smooth function $f:\mfd\to\R$ we have
\begin{equation*}
[\partial_s\gamma,\partial_t\gamma]f=\partial_s\gamma(\partial_t\gamma(f))-\partial_t\gamma(\partial_s\gamma(f))=\partial_s\partial_tf\circ\gamma-\partial_t\partial_sf\circ\gamma=0.
\end{equation*}
\item
We have
\begin{equation*}
\Rm(\partial_s\gamma,\partial_t\gamma)v=D_sD_tv-D_tD_sv-\nabla_{[\partial_s\gamma,\partial_t\gamma]}v,
\end{equation*}
where the last term is zero as shown in the previous point.
\qedhere
\end{enumerate}
\end{proof}

For a Lebesgue measurable function $v:I\to T\mfd$ on some measurable domain $I$ we denote its $L^p$-norm for $p\in[1,\infty]$ by
\begin{equation*}
\|v\|_{L^p}=\|\,|v|\,\|_{L^p(I)}.
\end{equation*}
If the domain is not clear from the context, we will write $\|v\|_{L^p(I)}$ instead.
Similarly, for a function $v\in\lift\gamma\mfd$ we introduce its H\"older norm of exponent $\alpha\in(0,1]$ by
\begin{equation*}
\|v\|_{C^{0,\alpha}}=\|v\|_{L^\infty}+\sup_{r,t\in[0,1]}\frac{|\partpt\gamma v(t)-\partpt\gamma v(r)|}{|t-r|^\alpha},
\end{equation*}
where $\partpt\gamma$ is parallel transport to $\gamma(0)$ along $\gamma$.

Finally, the Riemannian exponential with base point $p\in\mfd$ is denoted by $\exp_p:T_p\mfd\to\mfd$.
Letting $\rho$ be the injectivity radius of $\mfd$, the inverse of $\exp_p$ can be defined on the geodesic ball $B_\rho(p)\subset\mfd$ of radius $\rho$
and is called the Riemannian logarithm $\log_p:B_\rho(p)\to T_p\mfd$.
We close by providing some bounds on the Riemannian logarithm and exponential.

\begin{lemma}[Bounds on Riemannian logarithm and exponential]\label{thm:logDerivs}
Let $\mfd$ be a $C^{k+1}$-manifold, $k\geq2$, with injectivity radius $\rho$, sectional curvature bounded from above by $K$, and bounds $\|\Rm\|_\infty$ on $\Rm$ and $\|\nabla\Rm\|_\infty$ on $\nabla\Rm$. Then the bivariate mapping $\log(p,q)=\log_{p}q$ is in $C^{k}$ for $d(p,q)<\rho$. If $K>0$, assume additionally $d(p,q)\leq \frac{\pi}{2\sqrt{K}}$. The operator norms of the first derivatives of $\log$ satisfy
\begin{align*}
\|\partial_{1}\log_{p}q+\mathrm{Id}_{T_{p}M}\|&\leq \frac{\|\Rm\|_{L^{\infty}}}{2}d^{2}(p,q),\\
\|\partial_{2}\log_{p}q-\pi_{q\to p}\|&\leq \frac{\|\Rm\|_{L^{\infty}}}{2}d^{2}(p,q),
\end{align*}
where $\pi_{q\to p}$ denotes parallel transport along the geodesic from $q$ to $p$. The covariant second order derivatives, denoted by $D^{2}$, satisfy
\begin{align*}
\|D^{2}\log_{p}q\|\leq C(\|\Rm\|_\infty,\|\nabla\Rm\|_\infty) d(p,q)
\end{align*}
for a constant $C(\|\Rm\|_\infty,\|\nabla\Rm\|_\infty)$ depending only on $\|\Rm\|_\infty$ and $\|\nabla\Rm\|_\infty$.
Similarly, the bivariate mapping $\exp(p,v)=\exp_pv$ is in $C^k$, and for $|v|$ small enough, depending only on $\rho$, $K$, and $\|\Rm\|_\infty$, $\|\nabla\Rm\|_\infty$, it satisfies
\begin{align*}
\|\partial_{2}\exp_{p}v-\pi_{p\to q}\|&\leq C(\|\Rm\|_\infty,\|\nabla\Rm\|_\infty)|v|^2,\\
\|\partial_{2}^2\exp_{p}v\|&\leq C(\|\Rm\|_\infty,\|\nabla\Rm\|_\infty)|v|.
\end{align*}
\end{lemma}

The proof for the derivatives of the logarithm follows from Jacobi field estimates and can be found in \cite[Prop.\,A.1-A.2]{hardering_diss} (the estimates for the first order derivatives can also be found in \cite{karcher:1977}).
The estimates for the derivatives of the exponential are then straightforward applications of the inverse function theorem.

%

\section{Well-posedness and Euler--Lagrange equations of Riemannian spline interpolations}\label{sec:wellPosedness}

All throughout the article, $\gamma:[0,1]\to\mfd$ will be a given curve,
$k$ times differentiable ($k=2$ for linear and $k=4$ for cubic spline interpolation)
in the sense that $\gamma$ is continuous and $D_t^l\gamma$ is well-defined and bounded for $l=1,\ldots,k$.
Furthermore, $\gamma_h:[0,1]\to\mfd$ will denote the spline interpolation according to \cref{def:cubicSpline} or \cref{def:linearSpline} at points $t_i=ih$, $i=0,\ldots,N$, where $h=\frac1N$.

\begin{remark}[Boundary conditions for Riemannian cubic spline interpolation]
In \cref{def:cubicSpline} we chose to impose the so-called Hermite boundary conditions $\dot\gamma_h(t)=\dot\gamma(t)\text{ for }t=0,1$.
An alternative would be to solely require $\gamma_h=\gamma$ at all interpolation points.
This is known to result in so-called natural boundary conditions, essentially a vanishing acceleration of $\gamma_h$ at $t=0$ and $t=1$.
Unfortunately, though, existence of cubic spline interpolations with natural boundary conditions cannot be guaranteed due to lacking control of the curve velocity
(a simple example is provided in \cite{HeRuWi19}).
\end{remark}

We begin with straightforward a priori bounds that follow from the coercivity of the spline interpolation energy,
which is nothing else but $\|\dot\gamma_h\|_{L^2}^2$ in the case of linear splines and $\|D_t^2\gamma_h\|_{L^2}^2$ in case of cubic splines.
Boundedness of higher derivatives for the cubic spline interpolation will be proven later.

\begin{proposition}[A priori bounds]\label{thm:aPrioriBounds}
Let $\gamma:[0,1]\to\mfd$ be an absolutely continuous curve.
\begin{enumerate}
\item
If $\|\dot\gamma\|_{L^2}^2$ is bounded, then $\gamma$ is H\"older continuous with exponent $\frac12$, and for any $r\in[0,1]$ we have
\begin{equation*}
\|d(\gamma,\gamma(r))\|_{C^{0,\frac12}}
\leq2\|\dot\gamma\|_{L^2}.
\end{equation*}
\item\label{enm:aPrioriBoundsCubic}
If $\|D_t^2\gamma\|_{L^2}^2$ is bounded, then $\gamma$ is even H\"older continuously differentiable with exponent $\frac12$ and satisfies
\begin{align*}
\|\dot\gamma\|_{C^{0,\frac12}}
&\leq|\dot\gamma(0)|+2\|D_t^2\gamma\|_{L^2},\\
d(\gamma(t),\gamma(r))
&\leq|t-r|\|\dot\gamma\|_{L^\infty}\quad\text{for all }r,t\in[0,1].
\end{align*}
\end{enumerate}
\end{proposition}
\begin{proof}
\begin{enumerate}
\item
For any $t,r\in[0,1]$, by H\"older's inequality we have
\begin{equation*}
d(\gamma(t),\gamma(r))
\leq\int_{r}^t|\dot\gamma(\tau)|\,\d \tau
\leq\||\dot\gamma|\|_{L^2}\sqrt{|t-r|}
=\|\dot\gamma\|_{L^2}\sqrt{|t-r|}
\end{equation*}
so that $\gamma$ is indeed H\"older continuous with exponent $2$.
Furthermore, this implies
\begin{multline*}
\sup_{t\in[0,1]}d(\gamma(t),\gamma(r))
+\sup_{t,\tilde t\in[0,1]}\frac{|d(\gamma(t),\gamma(r))-d(\gamma(\tilde t),\gamma(r))|}{\sqrt{|t-\tilde t|}}
\leq\|\dot\gamma\|_{L^2}
+\sup_{t,\tilde t\in[0,1]}\frac{d(\gamma(t),\gamma(\tilde t))}{\sqrt{|t-\tilde t|}}
\leq2\|\dot\gamma\|_{L^2},
\end{multline*}
from which the estimate follows.
\item
Letting $\pi_{\gamma}$ denote parallel transport along $\gamma$ to $T_{\gamma(r)}\mfd$ we can analogously estimate
\begin{equation*}
|\pi_{\gamma}\dot\gamma(t)-\dot\gamma(r)|
=\left|\int_{r}^t\frac\d{\d \tau}\pi_{\gamma}\dot\gamma(\tau)\,\d \tau\right|
=\left|\int_{r}^t\pi_{\gamma}D_\tau\dot\gamma(\tau)\,\d \tau\right|
\leq\sqrt{|t-r|}\|\pi_{\gamma}D_t\dot\gamma\|_{L^2}
=\sqrt{|t-r|}\|D_t^2\gamma\|_{L^2}
\end{equation*}
by \cref{thm:rules}\eqref{enm:transportedDerivative} and H\"older's inequality. Thus
\begin{equation*}
\sup_{t\in[0,1]}|\dot\gamma(t)|
+\sup_{t,r\in[0,1]}\frac{|\pi_{\gamma}\dot\gamma(t)-\dot\gamma(r)|}{\sqrt{|t-r|}}
\leq|\dot\gamma(0)|+\sqrt t\|D_t^2\gamma\|_{L^2}+\|D_t^2\gamma\|_{L^2},
\end{equation*}
which is the first norm bound.
The second bound trivially follows from $d(\gamma(t),\gamma(r))\leq\int_r^t|\dot\gamma(\tau)|\,\d\tau$.
\qedhere
\end{enumerate}
\end{proof}

Obviously, an immediate consequence is that the linear spline interpolation $\gamma_h$ of $\gamma$, if it exists, is H\"older continuous with
\begin{equation}\label{eqn:closenessLinear}
d(\gamma_h(t),\gamma_h(t_i))
\leq2\sqrt{|t-t_i|}\|\dot\gamma\|_{L^2}\quad\text{for all }t\in[0,1],i\in\{1,\ldots,N\}
\end{equation}
and that the cubic spline interpolation, if it exists, is H\"older continuously differentiable with
\begin{align}\label{eqn:velocityCubic}
\|\dot\gamma_h\|_{C^{0,\frac12}}
&\leq|\dot\gamma(0)|+2\|D_t^2\gamma\|_{L^2},\\
d(\gamma_h(t),\gamma(t_i))
&\leq|t-t_i|\|\dot\gamma\|_{L^\infty}\quad\text{for all }t\in[0,1],i\in\{1,\ldots,N\}.
\label{eqn:closenessCubic}
\end{align}

We proceed by proving the well-posedness of spline interpolation, \cref{thm:mainResultLinear}\eqref{enm:wellPosednessLinear} and \cref{thm:mainResultCubic}\eqref{enm:wellPosednessCubic}.

\begin{proposition}[Well-posedness of Riemannian spline interpolation]\label{thm:wellPosedness}
Linear and cubic spline interpolations $\gamma_h$ of $\gamma:[0,1]\to\mfd$ exist.
Furthermore, they are unique if $\gamma$ is continuous (for linear spline interpolation) or has finite $\|D_t^2\gamma\|_{L^2}$ (for cubic spline interpolation)
and $h$ is small enough (depending on $\gamma$ and $\mfd$).
\end{proposition}
\begin{proof}
Obviously, a linear spline interpolation is obtained by minimizing for each $i=1,\ldots,N$ the energy
\begin{equation*}
\int_{t_{i-1}}^{t_i}|\dot\gamma_h|^2\,\d t
\end{equation*}
among all curve segments $\gamma_h:[t_{i-1},t_i]\to\mfd$ that satisfy $\gamma_h(t_{i-1})=\gamma(t_{i-1})$ and $\gamma_h(t_i)=\gamma(t_i)$.
Those curve segments are known to be exactly the constant speed-parameterized geodesics connecting $\gamma(t_{i-1})$ and $\gamma(t_i)$,
which always exist on a complete Riemannian manifold and which are unique if the end points are close enough to each other (depending on $\mfd$).
By uniform continuity of $\gamma$ the latter condition will be fulfilled for all $i=1,\ldots,N$ as soon as $h$ is small enough.

As for the well-posedness of the cubic spline interpolation,
consider a minimizing sequence $\gamma_h^n$, $n=1,2,\ldots$, of curves satisfying the interpolation conditions
such that the cubic spline energy $E[\gamma_h^n]=\|D_t^2\gamma_h^n\|_{L^2}^2$ converges monotonically to the infimum,
$E[\gamma_h^n]\to\inf\{E[\gamma_h]\,|\,\gamma_h(t_i)=\gamma(t_i)\text{ for }i=0,\ldots,N\text{ and }\dot\gamma_h(t)=\dot\gamma(t)\text{ for }t=0,1\}$.
Letting $\pi_{\tilde\gamma}$ denote parallel transport to $T_{\tilde\gamma(0)}\mfd$ along a curve $\tilde\gamma$, we have
\begin{equation*}
\|\tfrac\d{\d t}\pi_{\gamma_h^n}\dot\gamma_h^n\|_{L^2}^2
=\|\pi_{\gamma_h^n}D_t^2\gamma_h^n\|_{L^2}^2
=\|D_t^2\gamma_h^n\|_{L^2}^2
=E[\gamma_h^n]
\end{equation*}
so that the functions $\pi_{\gamma_h^n}\dot\gamma_h^n:[0,1]\to T_{\gamma(0)}\mfd$ have uniformly bounded $H^1$-seminorm.
Together with $\pi_{\gamma_h^n}\dot\gamma_h^n(0)=\dot\gamma_h^n(0)=\dot\gamma(0)$ it follows by Poincar\'e's inequality that the functions even have uniformly bounded $H^1$-norm
so that a subsequence converges weakly in $H^1$ (and thus strongly in $C^0([0,1])$) to some function $\beta:[0,1]\to T_{\gamma(0)}\mfd$.
Now define $\gamma_h:[0,1]\to\mfd$ as the solution to the ordinary differential equation
\begin{equation*}
\dot\gamma_h(t)=\pi_{\gamma_h}^{-1}\beta(t)
\end{equation*}
with initial value $\gamma_h(0)=\gamma(0)$, where we write $\pi_{\gamma_h}^{-1}\beta(t)$ for the parallel transport of $\beta(t)$ to $\gamma_h(t)$ along $\gamma_h$.
The solution $\gamma_h:[0,1]\to\mfd$ exists and is unique by the theorem of Picard and Lindel\"of, and it satisfies the interpolation conditions by Gronwall's lemma.
Indeed, let us express all quantities in local coordinates $\R^d$ and indicate this by a hat.
Further, let $\hat\Pi:[0,1]\to\R^{d\times d}$ denote the matrix representation in local coordinates of $\pi_{\gamma_h}^{-1}$.
Then $\hat\gamma_h$ solves the initial value problem
\begin{multline*}
\frac\d{\d t}{\hat\gamma_h\choose\hat\Pi}(t)
=F(t,\hat\gamma_h(t),\hat\Pi(t)),\;
{\hat\gamma_h\choose\hat\Pi}(0)
={\hat\gamma_h(0)\choose\mathrm{Id}}\\
\quad\text{for }
F:\R\times\R^d\times\R^{d\times d}\to\R^d\times\R^{d\times d},\;
F(t,\hat\gamma_h,\hat\Pi)={\hat\Pi\hat\beta(t)\choose-\Gamma_{\hat\gamma_h}(\hat\Pi\hat\beta(t),\hat\Pi)},
\end{multline*}
where $\Gamma_{\hat\gamma_h}:\R^d\times\R^d\to\R^d$ is the Christoffel operator at coordinates $\hat\gamma_h$
and the second entry of $F(t,\hat\gamma_h,\hat\Pi)$ refers to the matrix defined by $\R^d\ni\xi\mapsto\Gamma_{\hat\gamma_h}(\hat\Pi\hat\beta(t),\hat\Pi\xi)\in\R^d$.
The function $F$ is obviously Lipschitz so that the Picard--Lindel\"of theorem can be applied
(it is straightforward to show that $\hat\gamma_h$ and $\hat\Pi$ stay bounded for $t\in[0,1]$).
\notinclude{
\begin{multline*}
\frac\d{\d t}{\hat\gamma_h\choose\hat\alpha}(t)
=F(t,\hat\gamma_h(t),\hat\alpha(t)),\;
{\hat\gamma_h\choose\hat\alpha}(0)
={\hat\gamma_h(0)\choose\hat\beta}\\
\quad\text{for }
F:\R\times\R^d\times C^0([0,1])^d\to\R^d\times C^0([0,1])^d,\;
F(t,\hat v,\hat w)={\hat w(t)\choose\Gamma_{\hat v}(\hat w,\hat w(t))},
\end{multline*}
where $\Gamma_{\hat v}:\R^d\times\R^d\to\R^d$ is the Christoffel operator at coordinates $\hat v$
and the second entry of $F(t,\hat v,\hat w)$ refers to the function $s\mapsto\Gamma_{\hat v}(\hat w(s),\hat w(t))$.
The function $F$ is obviously Lipschitz so that the Picard--Lindel\"of theorem can be applied
(it is straightforward to show that $\hat\gamma_h$ and $\hat\alpha$ stay bounded for $t\in[0,1]$).
}
Since $\hat\gamma_h^n$ satisfies the same initial value problem, only with $\hat\beta$ replaced by $\widehat{\pi_{\gamma_h^n}\dot\gamma_h^n}$,
the curves $\hat\gamma_h^n$ converge uniformly on $[0,1]$ to $\hat\gamma_h$ by Gronwall's lemma, as do the corresponding $\hat\alpha$.
Consequently, $\hat\gamma_h$ satisfies the interpolation conditions.
Finally note that
\begin{equation*}
E[\gamma_h]
=\|\pi_{\gamma_h}D_t^2\gamma_h\|_{L^2}^2
=\|\tfrac\d{\d t}\pi_{\gamma_h}\dot\gamma_h\|_{L^2}^2
=\|\dot\beta\|_{L^2}^2
\leq\liminf_{n\to\infty}\|\tfrac\d{\d t}\pi_{\gamma_h^n}\dot\gamma_h^n\|_{L^2}^2
=\liminf_{n\to\infty}E[\gamma_h^n]
\end{equation*}
due to the sequential weak lower semi-continuity of the $L^2$-norm so that $\gamma_h$ minimizes $E$ under the interpolation constraints.

It remains to prove uniqueness of $\gamma_h$.
To this end we will show local convexity of the problem for $h$ small enough.
Indeed, for $h$ small enough the interpolation problem can equivalently be formulated in terms of functions
\begin{equation*}
v_i:[t_{i-1},t_i]\to T_{\gamma_i}\mfd,
\quad
v_i(t)=\log_{\gamma_i}\gamma_h(t),\;i=1,\ldots,N,
\end{equation*}
where we abbreviated $\gamma_i=\gamma(t_{i-1})$
(equivalently, $\gamma_h(t)=\exp_{\gamma_i}v_i(t)$ for $t\in[t_{i-1},t_i]$).
This is possible since $E[\gamma_h]\leq E[\gamma]$ and thus by \cref{thm:aPrioriBounds}\eqref{enm:aPrioriBoundsCubic},
\begin{equation*}
d(\gamma_h(t),\gamma_i)
\leq h\|\dot\gamma\|_{L^\infty}
\leq h\left(|\dot\gamma(0)|+2\|D_t^2\gamma\|_{L^2}\right)
\quad\text{for }t\in[t_{i-1},t_i]
\end{equation*}
so that for $h$ small enough the Riemannian logarithm in the definition of $v_i$ is uniquely defined and smooth.
Furthermore, again by \cref{thm:aPrioriBounds}\eqref{enm:aPrioriBoundsCubic}, the $v_i$ satisfy
\begin{align*}
\|[\partial_2\exp_{\gamma_i}v_i]\dot v_i\|_{C^{0,\frac12}}
&\leq C
\quad\text{for }C=|\dot\gamma(0)|+2\|D_t^2\gamma\|_{L^2},\\
\|v_i\|_{L^\infty}
&\leq Ch,
\end{align*}
for $i=1,\ldots,N$, which by \cref{thm:logDerivs} implies $\|\dot v_i\|_{L^\infty}\leq2C$ for $h$ small enough.
Now introduce
\begin{equation*}
J[v_1,\ldots,v_N]
=\sum_{i=1}^N\int_{t_{i-1}}^{t_i}\left|[\partial_2\exp_{\gamma_i}v_i(t)]\ddot v_i(t)+[\partial_2^2\exp_{\gamma_i}v_i(t)](\dot v_i(t),\dot v_i(t))\right|^2\d t,
\end{equation*}
which equals $E[\gamma_h]$.
Due to the above bounds on $\|v_i\|_{L^\infty}$ and $\|\dot v_i\|_{L^\infty}$ we can apply \cref{thm:logDerivs}
to obtain $E[\gamma]\geq E[\gamma_h]=J[v_1,\ldots,v_N]\geq\frac12\|\ddot v_i\|_{L^2}^2-1$ for $h$ small enough.
Thus, the cubic spline interpolation problem in terms of the $v_i$ can thus be written as
\begin{align*}
&\!\min\left\{J[v_1,\ldots,v_N]\,\middle|\,(v_1,\ldots,v_N)\in\mathcal A_h\right\}
\quad\text{with }\\
&\mathcal A_h=\big\{(v_1,\ldots,v_N)\in L^\infty((t_0,t_1))\times\ldots\times L^\infty((t_{N-1},t_N))\,\big|\,
\|v_i\|_{L^\infty}\leq Ch,\,\|\dot v_i\|_{L^\infty}\leq2C,\,\|\ddot v_i\|_{L^2}^2\leq2E[\gamma]+2,\\
&\hspace*{57ex}v_i(t_{i-1})=0,\,v_i(t_i)=\log_{\gamma_i}\gamma(t_i)\text{ for }i=1,\ldots,N\\
&\hspace*{52ex}\text{and }\dot v_1(0)=\dot\gamma(0),\,[\partial_2\exp_{\gamma(t_{N-1})}v_N(1)]\dot v_N(1)=\dot\gamma(1)\big\}.
\end{align*}
It is obvious that $\mathcal A_h$ is convex, so it remains to show convexity of $J$ on $\mathcal A_h$.
\notinclude{
Now consider a minimizing sequence $(v_1^k,\ldots,v_N^k)$, $k=1,2,\ldots$, of functions satisfying the interpolation constraints
such that $J[v_1^k,\ldots,v_N^k]\to\inf E$ monotonically as $k\to\infty$.
Without loss of generality we may assume $J[v_1^1,\ldots,v_N^1]\leq J[\log_{\gamma(0)}\gamma,\ldots,\log_{\gamma(t_{N-1})}\gamma]=\|D_t^2\gamma\|_{L^2}^2$.
Furthermore, for $h$ small enough, by \cref{thm:logDerivs} the energy is coercive in the sense
\begin{equation*}
J[v_1,\ldots,v_N]
\geq\frac12\sum_{i=1}^N\|\ddot v_i\|_{L^2}^2
\quad\text{whenever }(v_1,\ldots,v_N)\in\mathcal A_h
\end{equation*}
Thus, for a subsequence (not renamed) we have weak convergence $(v_1^k,\ldots,v_N^k)\rightharpoonup(v_1^\infty,\ldots,v_N^\infty)$ in $H^2((t_0,t_1))\times\ldots\times H^2((t_{N-1},t_N))$ as $k\to\infty$
and even $(v_1^k,\ldots,v_N^k)\to(v_1^\infty,\ldots,v_N^\infty)$ strongly in $C^{1,\frac12}([t_0,t_1])\times\ldots\times C^{1,\frac12}([t_{N-1},t_N])$ by Sobolev embedding.
The limit $(v_1,\ldots,v_N)$ thus also satisfies the interpolation conditions.
It remains to show lower semi-continuity of $J$ along the sequence and uniqueness of $(v_1^\infty,\ldots,v_N^\infty)$.
To this end we will show convexity of $J$ on $\mathcal A_h$.
It is obvious that the interpolation conditions in $J$ are convex, so it remains to show convexity of the sum of integrals.
}
Abbreviating $\gamma_v(t)=\exp_{\gamma_i}v_i(t)$ for $t\in[t_{i-1},t_i]$,
the G\^ateaux derivatives of $J$ in a direction $\phi\equiv(\varphi_1,\ldots,\varphi_N)$ with $\varphi_i(t_{i-1})=\varphi(t_i)=0$ and $\dot\varphi_1(0)=\dot\varphi_N(1)=0$ are
\begin{align*}
&\partial J[v_1,\ldots,v_N](\phi)\\
&=2\sum_{i=1}^N\int_{t_{i-1}}^{t_i}(D_t^2\gamma_v,\underbrace{[\partial_2\exp_{\gamma_i}v_i]\ddot\varphi_i+[\partial_2^2\exp_{\gamma_i}v_i](\ddot v_i,\varphi_i)+[\partial_2^3\exp_{\gamma_i}v_i](\dot v_i,\dot v_i,\varphi_i)+2[\partial_2^2\exp_{\gamma_i}v_i](\dot v_i,\dot\varphi_i)}_{U_i})\,\d t,\\
&\partial^2J[v_1,\ldots,v_N](\phi,\phi)\\
&=2\sum_{i=1}^N\int_{t_{i-1}}^{t_i}(U_i,U_i)+(D_t^2\gamma_v,\underbrace{2[\partial_2^2\exp_{\gamma_i}v_i](\varphi_i,\ddot\varphi_i)+[\partial_2^3\exp_{\gamma_i}v_i](\varphi_i,\varphi_i,\ddot v_i)+4[\partial_2^3\exp_{\gamma_i}v_i](\dot\varphi_i,\dot v_i,\varphi_i)}_{V_i^1}\\
&\hspace*{55ex}+\underbrace{[\partial_2^4\exp_{\gamma_i}v_i](\varphi_i,\varphi_i,\dot v_i,\dot v_i)+2[\partial_2^2\exp_{\gamma_i}v_i](\dot\varphi_i,\dot\varphi_i)}_{V_i^2})\,\d t.
\end{align*}
Now by \cref{thm:logDerivs},
\begin{align*}
\|U_i-\ddot\varphi_i\|_{L^2}
&\leq K\|v_i\|_{L^\infty}\left(\|v_i\|_{L^\infty}\|\ddot\varphi_i\|_{L^2}+\|\ddot v_i\|_{L^2}\|\varphi_i\|_{L^\infty}+\|\dot v_i\|_{L^\infty}\|\dot\varphi_i\|_{L^2}\right)
+K\|\dot v_i\|_{L^\infty}^2\|\dot\varphi_i\|_{L^2},\\
\|V_i^1+V_i^2\|_{L^2}
&\leq K\|v_i\|_{L^\infty}\left(\|\varphi_i\|_{L^\infty}\|\ddot\varphi_i\|_{L^2}+\|\dot\varphi_i\|_{L^\infty}\|\dot\varphi_i\|_{L^2}\right)\\
&\quad+K\left(\|\varphi_i\|_{L^\infty}^2\|\ddot v_i\|_{L^2}+\|\varphi_i\|_{L^\infty}\|\dot v_i\|_{L^\infty}\|\ddot\varphi_i\|_{L^2}+\|\dot v_i\|_{L^\infty}^2\|\varphi_i\|_{L^\infty}\|\varphi_i\|_{L^2}\right)
\end{align*}
for a constant $K>0$ depending only on $\mfd$.
Now recall that for $(v_1,\ldots,v_N)\in\mathcal A_h$ we have
\begin{equation*}
\|v_i\|_{L^\infty}\leq Ch,\quad
\|\dot v_i\|_{L^\infty}\leq2C,\quad
\|\ddot v_i\|_{L^2}^2\leq2\|D_t^2\gamma\|_{L^2}^2+2.
\end{equation*}
Furthermore, due to the homogeneous boundary conditions on each $\varphi_i$ there exists $c>0$ independent of $h$ or $\varphi_i$ such that
\begin{equation*}
\|\varphi_i\|_{L^\infty}\leq ch^{3/2}\|\ddot\varphi_i\|_{L^2},\quad
\|\dot\varphi_i\|_{L^\infty}\leq ch^{1/2}\|\ddot\varphi_i\|_{L^2},\quad
\|\dot\varphi_i\|_{L^2}\leq ch\|\ddot\varphi_i\|_{L^2}
\end{equation*}
(as follows from straightforward scaling arguments).
Summarizing, there is a constant $\kappa>0$ depending only on $\mfd$ and $\gamma$ such that
\begin{equation*}
\|U_i-\ddot\varphi_i\|_{L^2}
\leq\kappa h\|\ddot\varphi_i\|_{L^2},\quad
\|V_i^1+V_i^2\|_{L^2}
\leq\kappa h^{3/2}\|\ddot\varphi_i\|_{L^2}^2.
\end{equation*}
Furthermore, for $h$ small enough we have
\begin{multline*}
\|D_t^2\gamma_v\|_{L^2((t_{i-1},t_i))}^2
=\|[\partial_2\exp_{\gamma_i}v_i]\ddot v_i+[\partial_2^2\exp_{\gamma_i}v_i](\dot v_i,\dot v_i)\|_{L^2((t_{i-1},t_i))}^2\\
\leq2\|[\partial_2\exp_{\gamma_i}v_i]\ddot v_i\|_{L^2((t_{i-1},t_i))}^2+2\|[\partial_2^2\exp_{\gamma_i}v_i](\dot v_i,\dot v_i)\|_{L^2((t_{i-1},t_i))}^2
\leq4\|\ddot v_i\|_{L^2((t_{i-1},t_i))}^2+1
\leq 8\|D_t^2\gamma\|_{L^2}^2+9
\end{multline*}
by \cref{thm:logDerivs} and the previous estimates on $v_i$ so that
\begin{multline*}
\partial^2J[v_1,\ldots,v_N](\phi,\phi)
\geq2\sum_{i=1}^N\left(\frac12\|\ddot\varphi_i\|_{L^2}^2-\kappa^2h^2\|\ddot\varphi_i\|_{L^2}^2-\|D_t^2\gamma_v\|_{L^2((t_{i-1},t_i))}\kappa h^{\frac32}\|\ddot\varphi_i\|_{L^2}^2\right)\\
\geq(1-2\kappa^2h^2-\kappa h^{\frac32}(8\|D_t^2\gamma\|_{L^2}+9))\sum_{i=1}^N\|\ddot\varphi_i\|_{L^2}^2
>0
\end{multline*}
for $h$ small enough (depending just on $\mfd$ and $\gamma$).
Therefore, $J$ is strictly convex on $\mathcal A_h$.
\notinclude{
so that any minimizer on $\mathcal A_h$ (and thus everywhere) must be unique.
Furthermore, by Mazur's lemma there exists a sequence of convex combination coefficients $(a_1^k,\ldots,a_{N_k}^k)$, $k=1,2,\ldots$,
such that $\sum_{i=1}^{N_k}a_i^k(v_1^{k+i},\ldots,v_N^{k+i})\to(v_1^\infty,\ldots,v_N^\infty)$ strongly in $H^2((t_0,t_1))\times\ldots\times H^2((t_{N-1},t_N))$
and thus (up to a subsequence) all function values and first and second derivatives converge even pointwise almost everywhere.
Hence we have
\begin{multline*}
\liminf_{k\to\infty}J[v_1^k,\ldots,v_N^k]
=\lim_{k\to\infty}\sum_{i=1}^{N_k}a_i^kJ[v_1^{k+i},\ldots,v_N^{k+i}]
\geq\lim_{k\to\infty}E\left[\sum_{i=1}^{N_k}a_i^k(v_1^{k+i},\ldots,v_N^{k+i})\right]\\
\geq E\left[\lim_{k\to\infty}\sum_{i=1}^{N_k}a_i^k(v_1^{k+i},\ldots,v_N^{k+i})\right]
=J[v_1^\infty,\ldots,v_N^\infty],
\end{multline*}
where we used monotonicity of $J[v_1^k,\ldots,v_N^k]$ in the first equality, the convexity of $E$ in the first inequality,
and Fatou's lemma in the last inequality.
Thus, $(v_1^\infty,\ldots,v_N^\infty)$ is the unique minimizer.
}
\end{proof}

We close the section by deriving the Euler--Lagrange equations satisfied by the linear and cubic spline interpolations.

\begin{proposition}[Optimality conditions]\label{thm:optCond}
\begin{enumerate}
\item
A linear spline interpolation $\gamma_h$ satisfies the optimality conditions
\begin{gather*}
0=D_t^2\gamma_h\text{ on }(t_{k-1},t_{k}),\,k=1,\ldots,N,\\
\gamma_h=\gamma\text{ at }t_0,\ldots,t_N.
\end{gather*}
\item\label{enm:optCond}
A cubic spline interpolation $\gamma_h$ satisfies the optimality conditions
\begin{gather*}
0=D_t^4\gamma_h+\Rm(D_t^2\gamma_h,\dot\gamma_h)\dot\gamma_h\text{ on }(t_{k-1},t_{k}),\,k=1,\ldots,N,\\
D_t\gamma_h\text{ and }D_t^2\gamma_h\text{ exist and }\gamma_h=\gamma\text{ at }t_0,\ldots,t_N,\\
D_t\gamma_h=D_t\gamma\text{ at }t_0\text{ and }t_N.
\end{gather*}
\end{enumerate}
\end{proposition}
\begin{proof}
\begin{enumerate}
\item
Extend $\gamma_h$ to a family $\gamma_h:\R\times[0,1]\to\mfd$ of competing curves with $\gamma(0,\cdot)$ being the minimizer.
Then at $s=0$ we necessarily have
\begin{multline*}
0
=\frac12\frac\d{\d s}\int_0^1|\partial_t\gamma_h|^2\,\d t
=\int_0^1(\partial_t\gamma_h,D_s\partial_t\gamma_h)\,\d t
=\int_0^1(\partial_t\gamma_h,D_t\partial_s\gamma_h)\,\d t\\
=\sum_{k=1}^N\left[(\partial_t\gamma_h,\partial_s\gamma_h)\right]_{t_{k-1}}^{t_k}-\int_{t_{k-1}}^{t_k}(D_t^2\gamma_h,\partial_s\gamma_h)\,\d t
=-\sum_{k=1}^N\int_{t_{k-1}}^{t_k}(D_t^2\gamma_h,\partial_s\gamma_h)\,\d t
\end{multline*}
where we used \cref{thm:rules}, performed integration by parts using $\frac\d{\d t}(v,w)=(D_tv,w)+(v,D_tw)$
and where we exploited $\partial_s\gamma_h(t_k)=0$ due to the admissibility of the curve family $\gamma_h$.
Since $\gamma_h$ was extended arbitrarily, the result now follows from the fundamental lemma of the calculus of variations.
\item
Extend $\gamma_h$ to a family $\gamma_h:\R\times[0,1]\to\mfd$ of competing curves with $\gamma(0,\cdot)$ being the minimizer.
Then at $s=0$ we necessarily have
\begin{align*}
0
&=\frac12\frac\d{\d s}\int_0^1|D_t^2\gamma_h|^2\,\d t\\
&=\int_0^1(D_t^2\gamma_h,D_sD_t^2\gamma_h)\,\d t\\
&=\int_0^1(D_t^2\gamma_h,D_tD_s\partial_t\gamma_h+\Rm(\partial_s\gamma_h,\partial_t\gamma_h)\partial_t\gamma_h)\,\d t\\
&=\sum_{k=1}^N\left[(D_t^2\gamma_h,D_s\partial_t\gamma_h)\right]_{t_{k-1}}^{t_k}+\int_{t_{k-1}}^{t_k}(D_t^2\gamma_h,\Rm(\partial_s\gamma_h,\partial_t\gamma_h)\partial_t\gamma_h)-(D_t^3\gamma_h,D_s\partial_t\gamma_h)\,\d t\\
&=\sum_{k=1}^N\left[(D_t^2\gamma_h,D_t\partial_s\gamma_h)\right]_{t_{k-1}}^{t_k}+\int_{t_{k-1}}^{t_k}(\partial_s\gamma_h,\Rm(D_t^2\gamma_h,\partial_t\gamma_h)\partial_t\gamma_h)-(D_t^3\gamma_h,D_t\partial_s\gamma_h)\,\d t\\
&=\sum_{k=1}^N\left[(D_t^2\gamma_h,D_t\partial_s\gamma_h)\right]_{t_{k-1}}^{t_k}-\left[(D_t^3\gamma_h,\partial_s\gamma_h)\right]_{t_{k-1}}^{t_k}+\int_{t_{k-1}}^{t_k}(D_t^4\gamma_h+\Rm(D_t^2\gamma_h,\partial_t\gamma_h)\partial_t\gamma_h,\partial_s\gamma_h)\,\d t\\
&=\sum_{k=1}^{N-1}((D_t^-)^2\gamma_h(t_k)-(D_t^+)^2\gamma_h(t_k),D_t\partial_s\gamma_h(t_k))+\sum_{k=1}^N\int_{t_{k-1}}^{t_k}(D_t^4\gamma_h+\Rm(D_t^2\gamma_h,\partial_t\gamma_h)\partial_t\gamma_h,\partial_s\gamma_h)\,\d t,
\end{align*}
where $D_t^+$ and $D_t^-$ denote the left and right derivative.
Above, we again repeatedly exploited \cref{thm:rules} and performed integration by parts.
Furthermore we exploited $\partial_s\gamma_h(t_k)=0$ and $D_t\partial_s\gamma_h(t)=D_s\partial_t\gamma_h(t)=0$ at $t=t_0,t_N$ due to the admissibility of the curve family $\gamma_h$.
Since $\gamma_h$ was extended arbitrarily, the result now follows from the fundamental lemma of the calculus of variations.
\qedhere
\end{enumerate}
\end{proof}

\section{Two short convergence proofs in the Euclidean case}\label{sec:Euclidean}

Linear and cubic spline interpolation of curves in Euclidean space converge quadratically and quartically, repectively.
As our aim is to transfer this result to the Riemannian setting, we briefly provide the corresponding Euclidean proof here.
We first give a simple proof exploiting the piecewise polynomial structure of splines,
in particular that the second derivative of the cubic spline interpolation $\gamma_h$ is the $L^2$-best piecewise linear approximation of $\ddot\gamma$
(in the remainder of the article we will try to mimick this idea).
Afterwards we give a second, more direct but less elegant proof for cubic spline interpolation
which essentially just checks the stability of interpolation as well as its consistency (by comparing the Taylor expansions of spline interpolant and original curve).
In principle, an analogous approach should also be possible on a manifold,
but the corresponding stability and Taylor estimates would best be obtained in a chart and promise to be quite cumbersome.
Below we will denote the Sobolev space of integrable functions with essentially bounded derivatives up to fourth order by $W^{4,\infty}$.

\begin{theorem}[Spline convergence]\label{thm:EuclideanCase}
Let $\gamma\in W^{4,\infty}((0,1))$. Its linear spline interpolation $\gamma_h$ satisfies
\begin{equation*}
\|\gamma_h-\gamma\|_{L^p}\leq ch^2
\end{equation*}
for some $c>0$ depending only on $\gamma$ and $p$.
The corresponding cubic spline interpolation $\gamma_h$ satisfies
\begin{equation*}
\|\gamma_h-\gamma\|_{L^p}\leq ch^4
\end{equation*}
for some $c>0$ depending only on $\gamma$ and $p$.
\end{theorem}
\begin{proof}
For a function $f:[0,1]\to\R^d$ with zeros at $t_0,\ldots,t_N$ one has $\|f\|_{L^p}\leq Ch^2\|\ddot f\|_{L^p}$ for some $C>0$ depending on $p$.
Picking $f=\gamma_h-\gamma$ yields for the linear spline interpolation (which satisfies $(\ddot\gamma_h=0$)
\begin{equation*}
\|\gamma_h-\gamma\|_{L^p}
\leq Ch^2\|\ddot\gamma\|_{L^p}.
\end{equation*}
For the cubic spline interpolation one obtains
\begin{equation*}
\|\gamma_h-\gamma\|_{L^p}
\leq Ch^2\|\ddot\gamma_h-\ddot\gamma\|_{L^p}.
\end{equation*}
Now it turns out that $\ddot\gamma_h$ is the piecewise linear $L^2$-best approximation of $\ddot\gamma$.
Indeed, for any piecewise linear $w:[0,1]\to\R^d$ one obtains the Galerkin orthogonality
\begin{equation*}
\int_0^1(w,\ddot\gamma_h-\ddot\gamma)\,\d t
=\sum_{i=1}^N\left[(w,\dot\gamma_h-\dot\gamma)\right]_{t=t_{i-1}}^{t_i}-\int_{t_{i-1}}^{t_i}(\dot w,\dot\gamma_h-\dot\gamma)\,\d t
=\left[(w,\dot\gamma_h-\dot\gamma)-(\dot w,\gamma_h-\gamma)\right]_{t=t_{i-1}}^{t_i}
=0
\end{equation*}
from two integrations by parts, exploiting that $\gamma_h-\gamma=0$ at the interpolation points, that $\dot\gamma_h$ exists by \cref{thm:optCond}, and that $\dot\gamma_h-\dot\gamma=0$ at $t=0$ and $t=1$.
This immediately implies (e.\,g.\ via C\'ea's lemma) that $\ddot\gamma_h$ is the $L^2$-projection or -best approximation of $\ddot\gamma$.
Thus, by \cite[Lemma\,4.3]{DoDuWa75} (or by the $L^\infty$-boundedness of the $L^2$-projection onto splines \cite{dB76} and the resulting best approximation argument as in \cite[Thm.\,3.3.7]{Ci78}) we have
\begin{equation*}
\|\ddot\gamma_h-\ddot\gamma\|_{L^\infty}
\leq Ch^2\|\gamma\|_{W^{4,\infty}}
\end{equation*}
for some $C>0$.
\end{proof}

The alternative, more direct proof for cubic spline interpolation just uses stability and fourth order consistency of cubic spline interpolation via Taylor expansion.

\begin{proof}
The piecewise cubic spline interpolation can be written as a linear combination of shifted and scaled versions of the cubic B-spline basis function
\begin{equation*}
B(t)=\begin{cases}
(3|t|^3-6t^2+4)/6&\text{if }|t|\leq1,\\
(2-|t|)^3/6&\text{if }|t|\in(1,2],\\
0&\text{else.}
\end{cases}
\end{equation*}
Writing $\gamma_h$ as a linear combination
\begin{equation*}
\gamma_h(t)=\sum_{i=-1}^{N+1}x_iB(\tfrac th-i)
\end{equation*}
of these basis functions, the interpolation constraints translate into the linear system
\begin{equation*}
Ax=b
\quad\text{with }
A=\frac16\begin{psmallmatrix}\frac12&1\\1&4&1\\&1&4&1\\&&\ddots&\ddots&\ddots\\&&&1&4&1\\&&&&1&\frac12\end{psmallmatrix},\,
x=\begin{psmallmatrix}x_{-1}\\\vdots\\x_{N+1}\end{psmallmatrix},\,
b=\begin{psmallmatrix}\tfrac{\gamma(t_0)}4-\tfrac{\dot\gamma(0)h}{12}\\\gamma(t_0)\\\vdots\\\gamma(t_N)\\\tfrac{\gamma(t_N)}4+\tfrac{\dot\gamma(1)h}{12}\end{psmallmatrix}.
\end{equation*}
Now for a point $t\in[t_i,t_{i+1}]$ we abbreviate $a_i=B(\tfrac th-i)$ (thus all $a_i$ are zero except for $a_{i-1},\ldots,a_{i+2}$) so that
\begin{equation*}
\gamma_h(t)-\gamma(t)=a^Tx-\gamma(t)=a^TA^{-1}b-\gamma(t)
\quad\text{with }
a=(a_{-1},\ldots,a_{N+1})^T.
\end{equation*}
Now define $p(r)=\gamma(t)+\dot\gamma(t)(r-t)+\frac12\ddot\gamma(t)(r-t)^2+\frac16\dddot\gamma(t)(r-t)^3$ to be the third order Taylor expansion of $\gamma$ and notice
\begin{equation*}
b=b_1+b_2=\begin{psmallmatrix}p(t_0)/4-\dot p(0)h/12\\p(t_0)\\\vdots\\p(t_N)\\p(t_N)/4+\dot p(1)h/12\end{psmallmatrix}+\begin{psmallmatrix}O(|0-t|^4)\\O(|t_0-t|^4)\\\vdots\\O(|t_N-t|^4)\\O(|1-t|^4)\end{psmallmatrix}.
\end{equation*}
Since $a^TA^{-1}b_1=p(t)=\gamma(t)$, we arrive at
\begin{equation*}
|\gamma_h(t)-\gamma(t)|
=|a^TA^{-1}b_2|
=\left|\sum_{j=i-1}^{i+2}a_jA^{-1}_{j,:}b_2\right|
\leq\sum_{j=i-1}^{i+2}|A^{-1}_{j,:}b_2|,
\end{equation*}
where $A^{-1}_{j,:}=(q_{-1},q_0,\ldots,q_{N+1})$ denotes the $j$th row of $A^{-1}$ and thus necessarily satisfies
\begin{gather*}
q_{-1}=-2q_0
\quad\text{and}\quad
q_k=-(q_{k-2}+4q_{k-1}),\,k=1,\ldots,j,\\
q_{N+1}=-2q_N
\quad\text{and}\quad
q_k=-(q_{k+2}+4q_{k+1}),\,k=N-1,\ldots,j.
\end{gather*}
By induction this immediately implies
\begin{equation*}
q_kq_{k-1}<0,\,k=0,\ldots,N+1,
\qquad
|q_k|\begin{cases}
\geq|q_{k-1}|,\,k=1,\ldots,j,\\
\leq|q_{k-1}|,\,k=j+1,\ldots,N,
\end{cases}
\quad\text{ and }
|q_k|\leq3^{-|k-j|}|q_j|,\,k=0,\ldots,N.
\end{equation*}
Furthermore, the equation $1=q_{j-1}+4q_j+q_{j+1}$ together with the fact that $q_{j-1},q_{j+1}$ have a different sign than $q_j$ and are at least by the factor $3$ smaller in absolute value implies
\begin{equation*}
|q_j|\leq1
\quad\text{and thus}\quad
|q_k|\leq3^{-|k-j|},\,k=0,\ldots,N.
\end{equation*}
Summarizing, we obtain
\begin{equation*}
|\gamma_h(t)-\gamma(t)|
\leq\sum_{j=i-1}^{i+2}|A^{-1}_{j,:}b_2|
\leq\sum_{j=i-1}^{i+2}\left|\sum_{k=-1}^{N+1}3^{-|k-j|}O(|t_k-t|^4)\right|
\leq4h^4\sum_{k=-\infty}^\infty3^{-|k-j|}(|k-j|+1)^4
\leq Ch^4
\end{equation*}
for a constant depending only on the maximum fourth derivative of $\gamma$.
\end{proof}

\begin{corollary}[Perturbed spline]\label{thm:perturbedSpline}
Let $g:[0,1]\to\R^d$ be twice differentiable with $g(t_i)=\dot g(0)=\dot g(1)=0$ for $i=0,\ldots,N$
and $g\in W^{4,\infty}((t_{i-1},t_i))$ with $\|\partial_t^4g\|_{L^\infty((t_{i-1},t_i))}\leq C$ for $i=1,\ldots,N$.
Then $\|g\|_{L^\infty}\leq ch^4$ with a constant $c>0$ depending only on $C$.
\end{corollary}
\begin{proof}
Define $\gamma:[0,1]\to\R^d$ to be the curve in $W^{4,\infty}((0,1))$ with $\gamma(0)=\dot\gamma(0)=\gamma(1)=\dot\gamma(1)=0$
whose fourth weak derivative is $f\in L^\infty((0,1))$ with $f=\partial_t^4g$ in each $(t_{i-1},t_i)$,
that is, we set
\begin{align*}
\tilde\gamma(t)&=\int_0^t\int_0^s\int_0^r\int_0^qf(p)\,\d p\,\d q\,\d r\,\d s,\\
\gamma(t)&=\tilde\gamma(t)+(\dot{\tilde\gamma}(1)-3\tilde\gamma(1))t^2+(2\tilde\gamma(1)-\dot{\tilde\gamma}(1))t^3.
\end{align*}
Letting now $\gamma_h$ be the cubic spline interpolation of $\gamma$, we obviously have $\gamma-\gamma_h=g$
and thus $\|g\|_{L^\infty}\leq ch^4$, where the constant only depends on the maximum fourth derivative of $\gamma$ (which is bounded by $C$).
\end{proof}

\section{The approximation properties of Riemannian cubic splines}\label{sec:convergence}
In this section we transfer the proof of \cref{thm:EuclideanCase} for the Euclidean case to the Riemannian setting.
Essentially, we will proceed in three steps:
we will first derive a generalization of the estimate $\|\gamma_h-\gamma\|_{L^p}\leq Ch^2\|\ddot\gamma_h-\ddot\gamma\|_{L^p}$,
then we will introduce a linear interpolation operator $\linInt$ of vector fields
as a preparation for estimating $\|\ddot\gamma_h-\ddot\gamma\|_{L^p}$ (which we think of as some kind of best approximation error) by $\|\ddot\gamma_h-\linInt\ddot\gamma\|_{L^p}$,
and finally we will derive boundedness of $\|D_t^2\gamma_h\|_{L^\infty}$ using the $L^2$-best approximation idea from \cref{thm:EuclideanCase}
and bootstrap this $L^2$-best approximation argument with the help of the bound on $\|D_t^2\gamma_h\|_{L^\infty}$
to obtain a generalization of the Euclidean estimate $\|\ddot\gamma_h-\ddot\gamma\|_{L^\infty}\leq Ch^2\|\gamma\|_{W^{4,\infty}}$.

\subsection{Quadratic convergence rate exploiting local interpolation}
Just like in the proof of \cref{thm:EuclideanCase} for the Euclidean case we first derive a quadratic dependence of the spline interpolation error on $h$.

\begin{lemma}[Approximation lemma]\label{thm:approximationLemma}
Let $\gamma:[0,1]\to\mfd$ be a curve and $v\in\lift\gamma\mfd$ a vector field along the curve with $v(t_k)=v(t_{k+1})=\ldots=v(t_{k+l})=0$,
then for $p\in[1,\infty]$ and $I=(t_k,t_{k+l})$ we have
\begin{equation*}
\|v\|_{L^p(I)}
\leq ch^{l+1}\|D_t^{l+1}v\|_{L^p(I)}
\end{equation*}
for a constant $c$ depending only on $l$ and $p$.
\end{lemma}
\begin{proof}
Let $\pi_\gamma:T_{\gamma(t)}\mfd\to T_{\gamma(0)}\mfd$ denote the parallel transport along the curve $\gamma$.
Then $\pi_\gamma v:[0,1]\to T_{\gamma(0)}\mfd$ satisfies $\pi_\gamma v(t_k)=\ldots=\pi_\gamma v(t_{k+l})=0$
so that the standard approximation result for curves in Euclidean space \cite[Thm.\,3.1.5]{Ci78} or a straightforward scaling argument can be applied, giving
\begin{equation*}
\|v\|_{L^p(I)}
=\|\pi_\gamma v\|_{L^p(I)}
\leq ch^{l+1}\|\partial_t^{l+1}(\pi_\gamma v)\|_{L^p(I)}
=ch^{l+1}\|D_t^{l+1}v\|_{L^p(I)}
\end{equation*}
by \cref{thm:rules}\eqref{enm:transportedDerivative}.
\end{proof}

This immediately implies the generalization of the Euclidean estimate $\|\gamma_h-\gamma\|_{L^p}\leq Ch^2\|\ddot\gamma_h-\ddot\gamma\|_{L^p}$.

\begin{proposition}[Quadratic convergence]\label{thm:quadraticConvergence}
For any $p\in[1,\infty]$ the linear and cubic spline interpolation of an absolutely continuous curve $\gamma:[0,1]\to\mfd$ satisfy
\begin{equation*}
\|d(\gamma_h,\gamma)\|_{L^p(I)}
\leq ch^2\|D_t^2\log_{\gamma_h}\!\gamma\|_{L^p(I)}
\end{equation*}
on each interval $I=(t_{i-1},t_i)$
for a constant $c>0$ independent of $\gamma$, $\gamma_h$, or $h$, for
all $h$ small enough (both depending only on $\mfd$ and $\gamma$).
\end{proposition}
\begin{proof}
Simply apply \cref{thm:approximationLemma} with $l=1$ to $v=\log_{\gamma_h}\!\gamma$ and note $d(\gamma_h,\gamma)=|\log_{\gamma_h}\!\gamma|_{\gamma_h}$ as long as $h$ is small enough for $v$ to be well-defined (follows from \cref{thm:aPrioriBounds} and continuity of $\gamma$).
\end{proof}

The full convergence result of \cref{thm:mainResultCubic,thm:mainResultLinear} for linear and cubic spline interpolation will follow from estimating the right-hand side of the previous inequality.
To this end we expand
\begin{multline*}
D_t^2\log_{\gamma_h}\!\gamma
=\partial_2\log_{\gamma_h}\!\gamma\,(D_t^2\gamma)+\partial_1\log_{\gamma_h}\!\gamma\,(D_t^2\gamma_h)
+\partial_2^2\log_{\gamma_h}\!\gamma\,(\dot\gamma,\dot\gamma)
+\partial_2\partial_1\log_{\gamma_h}\!\gamma\,(\dot\gamma,\dot\gamma_h)\\
+\partial_1\partial_2\log_{\gamma_h}\!\gamma\,(\dot\gamma_h,\dot\gamma)
+\partial_1^2\log_{\gamma_h}\!\gamma\,(\dot\gamma_h,\dot\gamma_h).
\end{multline*}
Note that at this point, \cref{thm:quadraticConvergence} can already be turned into a quadratic convergence result (which for linear splines already is the optimal rate),
which is briefly done in the remainder of this paragraph.
Indeed, with \cref{thm:logDerivs} and the boundedness of $\gamma$ and $\gamma_h$ from \cref{thm:aPrioriBounds}
or more specifically from \eqref{eqn:closenessLinear}-\eqref{eqn:closenessCubic} we can estimate
\begin{equation*}
\|D_t^2\log_{\gamma_h}\!\gamma\|_{L^p(I)}
\leq C\left(\|D_t^2\gamma_h\|_{L^p(I)}+\|D_t^2\gamma\|_{L^p(I)}+\|(|\dot\gamma|+|\dot\gamma_h|)^2\|_{L^p(I))}\right)
\end{equation*}
for some $C>0$ if $h$ is small enough (depending on $\gamma$ and $\mfd$), which leads to the following result.

\begin{corollary}[Quadratic $L^1$- and $L^2$-convergence]\label{thm:quadraticConvergenceI}
Let $\gamma:[0,1]\to\mfd$ be twice differentiable, then its linear spline interpolation $\gamma_h$ satisfies
\begin{equation*}
\|d(\gamma_h,\gamma)\|_{L^1}
\leq ch^2
\end{equation*}
for some constant $c>0$ and all $h$ small enough (both depending only on $\mfd$ and $\gamma$).
If $\gamma$ is four times differentiable, then its cubic spline interpolation $\gamma_h$ satisfies
\begin{equation*}
\|d(\gamma_h,\gamma)\|_{L^2}
\leq ch^2
\end{equation*}
for some constant $c>0$ and all $h$ small enough (again both depending only on $\mfd$ and $\gamma$).
\end{corollary}
\begin{proof}
By \cref{thm:quadraticConvergence}, for the linear spline we have
\begin{multline*}
\|d(\gamma_h,\gamma)\|_{L^1}
\leq ch^2\sum_{i=1}^N\|D_t^2\log_{\gamma_h}\!\gamma\|_{L^1((t_{i-1},t_i))}
\leq ch^2\sum_{i=1}^NC\left(\|D_t^2\gamma\|_{L^1((t_{i-1},t_i))}+\||\dot\gamma|+|\dot\gamma_h|\|_{L^2((t_{i-1},t_i))}^2\right)\\
=cCh^2\left(\|D_t^2\gamma\|_{L^1}+\||\dot\gamma|+|\dot\gamma_h|\|_{L^2}^2\right)
\leq cCh^2\left(\|D_t^2\gamma\|_{L^1}+2\|\dot\gamma\|_{L^2}^2+2\|\dot\gamma_h\|_{L^2}^2\right)
\leq cCh^2\left(\|D_t^2\gamma\|_{L^1}+4\|\dot\gamma\|_{L^2}^2\right),
\end{multline*}
where we exploited the optimality condition $D_t^2\gamma_h=0$ as well as $\|\dot\gamma_h\|_{L^2}^2\leq\|\dot\gamma\|_{L^2}^2$.
Similarly, for the cubic spline we have
\begin{multline*}
\|d(\gamma_h,\gamma)\|_{L^2}
\leq ch^2\|D_t^2\log_{\gamma_h}\!\gamma\|_{L^2}
\leq cCh^2\left(\|D_t^2\gamma\|_{L^2}+\|D_t^2\gamma_h\|_{L^2}+\|(|\dot\gamma|+|\dot\gamma_h|)^2\|_{L^2}\right)\\
\leq cCh^2\left(2\|D_t^2\gamma\|_{L^2}+(\|\dot\gamma\|_{L^\infty}+\|\dot\gamma_h\|_{L^\infty})^2\right)
\leq cCh^2\left(2\|D_t^2\gamma\|_{L^2}+(\|\dot\gamma\|_{L^\infty}+|\dot\gamma(0)|+2\|D_t^2\gamma\|_{L^2})^2\right),
\end{multline*}
where we exploited $\|D_t^2\gamma_h\|_{L^2}\leq\|D_t^2\gamma\|_{L^2}$ as well as \eqref{eqn:velocityCubic}.
\end{proof}

To improve the previous convergence estimate for linear spline interpolation to the final estimate of \cref{thm:mainResultLinear} we require the following uniform bound on $\dot\gamma_h$,
which immediately follows from properties of unit speed parameterized geodesics.

\begin{proposition}[Boundedness of first derivative]\label{thm:boundFirstDeriv}
The linear spline interpolation $\gamma_h$ of a curve $\gamma$ satisfies
\begin{equation*}
\|\dot\gamma_h\|_{L^\infty}
\leq\|\dot\gamma\|_{L^\infty}.
\end{equation*}
\end{proposition}
\begin{proof}
Recall that between $t_{i-1}$ and $t_i$ the curve $\gamma_h$ is a shortest unit speed parameterized geodesic between $\gamma(t_{i-1})$ and $\gamma(t_i)$,
thus on $(t_{i-1},t_i)$ we have
\begin{equation*}
|\dot\gamma_h|
=d(\gamma(t_{i-1}),\gamma(t_i))/h
\leq\int_{t_{i-1}}^{t_i}|\dot\gamma|\,\d t/h
\leq\|\dot\gamma\|_{L^\infty}.
\qedhere
\end{equation*}
\end{proof}

Together with \cref{thm:boundFirstDeriv} for linear spline interpolation and $\|D_t^2\gamma_h\|_{L^\infty}\leq3\|D_t^2\gamma\|_{L^\infty} + Ch$ for cubic spline interpolation, derived later in \cref{thm:BoundSecondDeriv}, which by a bootstrapping argument makes use of \cref{thm:quadraticConvergenceI},
one finally arrives at the following quadratic convergence estimate.

\begin{corollary}[Quadratic $L^\infty$-convergence]\label{thm:quadraticConvergenceII}
The linear and cubic spline interpolation satisfy
\begin{equation*}
\|d(\gamma_h,\gamma)\|_{L^\infty}
\leq ch^2
\end{equation*}
for some constant $c>0$ depending only on $\|\dot\gamma\|_{L^\infty}$, $\|D_t^2\gamma\|_{L^\infty}$, and $\mfd$, for $h$ small enough depending on $\mfd$ and $\gamma$.
\end{corollary}
\begin{proof}
Abbreviating $I_i=(t_{i-1},t_i)$, by \cref{thm:quadraticConvergence} and \cref{thm:boundFirstDeriv} or \cref{thm:BoundSecondDeriv} we have
\begin{align*}
\|d(\gamma_h,\gamma)\|_{L^\infty}
&\leq ch^2\max_{i=1,\ldots,N}\|D_t^2\log_{\gamma_h}\!\gamma\|_{L^\infty(I_i)}\\
&\leq ch^2\max_{i=1,\ldots,N}C\left(\|D_t^2\gamma\|_{L^\infty(I_i)}+\|D_t^2\gamma_h\|_{L^\infty(I_i)}+\||\dot\gamma|+|\dot\gamma_h|\|_{L^\infty(I_i)}^2\right)\\
&\leq \tilde Ch^2\left(\|D_t^2\gamma\|_{L^\infty}+Ch+\|\dot\gamma\|_{L^\infty}^2\right),
\end{align*}
for some $\tilde C>0$, where in the case of linear splines we exploited that $D_t^2\gamma_h=0$.
\end{proof}

\subsection{A linear interpolation operator of vector fields}
As mentioned before, the guiding idea of our approach is to interpret $D_t^2\gamma_h$ as an approximation to $D_t^2\gamma$
which is (maybe up to a constant factor) as good as the $L^2$-best approximation or a piecewise linear interpolation.
To formalize this, below we introduce a piecewise linear interpolation operator of vector fields along curves
and analyse its well-posedness and approximation properties.

\begin{definition}[Linear interpolation operator]\label{def:linInt}
	Let $\gamma:[0,1]\to\mfd$ be a piecewise differentiable and globally continuous curve and $v\in\lift\gamma\mfd$ a continuous vector field along $\gamma$.
	We define the piecewise linear interpolation $\linInt v$ of $v$ as solution $w$ to the linear boundary value problem
	\begin{gather}\label{eq:jacobi}
	D_t^2w+\Rm(w,\dot\gamma)\dot\gamma=0\text{ on }(t_{k-1},t_k),\,k=1,\ldots,N,\\
	w(t_k)=v(t_k),\,k=0,\ldots,N.\notag
	\end{gather}
\end{definition}

\begin{lemma}[Well-posedness and boundedness of linear interpolation]\label{thm:LinearInterpolation}
For $h$ small enough depending on $\mfd$ and $\|\dot\gamma\|_{L^\infty}$, $\linInt$ is well-defined, that is, the boundary value problem \eqref{eq:jacobi} has a unique solution.
Furthermore, for every continuous vector field $v\in\lift\gamma\mfd$ we have 
\begin{align*}
\|\linInt v\|_{L^\infty} &\leq2\max_{k=0,\ldots,N}|v(t_{k})|\leq 2\|v\|_{L^\infty}.
\end{align*}
\end{lemma}
\begin{proof}
	As the boundary value problem~\eqref{eq:jacobi} is a linear system of second order ordinary differential equations, we obtain existence of solutions by standard ordinary differential equation theory.
	Uniqueness and the estimate on the norm follow similarly to Rauch comparison estimates for classical Jacobi fields on each interval $(t_{k-1},t_{k})$.
	In detail, assume that the scalar curvature of $\mfd$ is bounded from above by $\mu\geq 0$, and let $h$ be small enough such that $h\|\dot\gamma\|_{L^\infty}^{2}\leq\frac{\pi}{2\mu}$.
	Let $w$ be defined on $(t_{k-1},t_{k})$ by \eqref{eq:jacobi} with $w(t_{k-1})=0$ and $w(t_{k})=\hat{v}\in T_{\gamma(t_{k})}\mfd$.
	Set
	\begin{align*}f(t)=\partial_t|w|(t_{k-1})\;\tfrac{2h}{\pi}\sin\left(\tfrac{\pi(t-t_{k-1})}{2h}\right)\quad\text{for }t\in [t_{k-1},t_{k}],
	\end{align*}
	then
	\begin{gather*}
		\partial_t^2f+\tfrac\pi{2h} f=0\text{ on }(t_{k-1},t_{k})
		\quad\text{with}\\
		f(t_{k-1})=0, \qquad \partial_t f(t_{k-1})=\partial_t|w|(t_{k-1}), \quad\text{and }0< f(t)< f(t_{k}) \quad\text{for }t\in(t_{k-1},t_{k}).
	\end{gather*}
	As long as $w(t)\neq 0$, abbreviating by $K(w,\dot\gamma)=(w,\Rm(w,\dot\gamma)\dot\gamma)/\left(|w|^{2}|\dot \gamma|^{2}-( w,\dot \gamma)^{2}\right)$ the sectional curvature, we have
	\begin{align*}
		\partial_t^2|w| &=-\frac{1}{|w|}K(w,\dot \gamma)\left(|w|^{2}|\dot \gamma|^{2}-( w,\dot \gamma)^{2}\right) + \frac{1}{|w|^{3}}\left(|w|^{2}|D_tw|^{2}-( w,D_tw)^{2}\right)
		\geq -\mu|\dot \gamma|^{2}|w|
	\end{align*}
	and thus
	\begin{align*}
		\partial_t\left(f\partial_t|w|-|w|\partial_t f\right)=f\partial_t^2|w|-|w|\partial_t^2f\geq 0.
	\end{align*}
	As $(f\partial_t|w|-|w|\partial_t f)(t_{k-1})=0$, this implies
	\begin{align*}
		\partial_t\left(\tfrac{|w|}{f}\right)=\frac{1}{f^{2}}\left(f\partial_t|w|-|w|\partial_t f\right)\geq 0
	\end{align*}
	as long as $w(t)\neq 0$, and thus for $s>t>t_{k-1}$,
	\begin{align*}
		\frac{|w|(s)}{f(s)}\geq\frac{|w|(t)}{f(t)}\geq\lim_{r\searrow t_{k-1}}\frac{|w|(r)}{f(r)}=1,
	\end{align*}
	where the limit follows by de l'H\^opital's rule.
	This shows in particular, that $|w|$ cannot have a zero before $f$ does, and thus the estimate
	\begin{align*}
		|w|(t)\leq  \frac{f(t)}{f(t_{k})}|w|(t_{k})\leq |\hat{v}|
	\end{align*}
	holds for all $t\in [t_{k-1},t_{k}]$. Together with the linearity of \eqref{eq:jacobi}, this stability result implies uniqueness of the solution to the boundary value problem as well as (using that the same estimate is obtained when the roles of $t_{k-1}$ and $t_k$ are swapped) the estimate
	\begin{equation*}
		\|\linInt v\|_{L^\infty}\leq 2\max_{k=0,\ldots,N}|v(t_{k})|\leq 2\|v\|_{L^\infty}.
		\qedhere
	\end{equation*}
\end{proof}

\begin{lemma}[Approximation properties of linear interpolation]\label{thm:propertiesLinearInterpolation}
	Let $h$ be small enough such that \cref{thm:LinearInterpolation} applies, and let $p\in[1,\infty]$. There exists a constant $c>0$ depending on $p$ and $\mfd$ such that for any vector field $v\in\lift\gamma\mfd$ the estimate
	\begin{align*}
		\|D_t^{2}\linInt v\|_{L^p(I_{k})}&\leq c\|\dot\gamma\|_{L^\infty(I_{k})}^{2}\|\linInt v\|_{L^p(I_{k})} 
	\end{align*}
	holds on each interval $I_{k}=(t_{k},t_{k+1})$.
	For $p>q\geq1$ and $\alpha=(q^{-1}-p^{-1})$ there exist constants $c>0$ depending on $p$, $q$ and $\mfd$ such that the inverse estimates
	\begin{align*}
		\|\linInt v\|_{L^{p}}&\leq c\left(\|\dot \gamma\|^{\alpha}_{L^{\infty}}+h^{-\alpha}\right)\|\linInt v\|_{L^{q}},\\
		\|D_t\linInt v\|_{L^p}&\leq c\left(\|\dot\gamma\|_{L^\infty}^{1+\alpha}+h^{-1-\alpha}\right)\|\linInt v\|_{L^{q}}
	\end{align*}
	hold. If $\|D_t^lv\|_{L^p}$ is bounded for $l=0,1,2$, we further have the interpolation error estimate
	\begin{align*}
		\|v-\linInt v\|_{L^{p}}
		\leq  c\;h^{2}\left(\|D_t^{2}v\|_{L^{p}} + \|\dot\gamma\|_{L^{\infty}}^{2}\|v\|_{L^{\infty}}\right).
	\end{align*}
\end{lemma}
\begin{proof}
	The estimate for $D_t^{2}\linInt v$ follows on each interval $I_{k}$ directly from the definition of $\linInt$ with $c=\|\Rm\|_{L^{\infty}}$.
	The inverse estimates follow from the scaled Gagliardo--Nirenberg--Sobolev inequality and the estimates on the second derivatives,
	\begin{align*}
		\|D_t^{j}\linInt v\|_{L^{p}(I_{k})}&\leq c\|D_t^{2}\linInt v\|_{L^{q}(I_{k})}^{\beta}
		\|\linInt v\|_{L^{q}(I_{k})}^{1-\beta} + ch^{-j+p^{-1}-q^{-1}}\|\linInt v\|_{L^{q}(I_{k})}\\
		&\leq c\left(\|\dot\gamma\|_{L^\infty}^{2\beta} + h^{-j+p^{-1}-q^{-1}}\right)\|\linInt v\|_{L^{q}(I_{k})}
	\end{align*}
	with $p^{-1}=j-2\beta+q^{-1}$ and $j=0,1$.
	The global estimates are now obtained by taking the $p$th power, summing over all subintervals $I_{k}$, and exploiting that $\sum_ia_i^{p/q}\leq(\sum_ia_i)^{p/q}$ for any vector of nonnegative values $a_i$.
	The interpolation error follows from \cref{thm:approximationLemma} with $l=1$, the estimate for $D_t^2\linInt v$ on each $I_k$, and \cref{thm:LinearInterpolation}.
\end{proof}

\begin{remark}[Vector field interpolation along splines]
	We will always apply the vector field interpolation $\linInt$ along the cubic spline interpolations $\gamma_{h}$ of some continuous curve $\gamma$. The dependence on $\gamma_{h}$ in the estimates of \cref{thm:LinearInterpolation} and \cref{thm:propertiesLinearInterpolation} only comprises bounds on $\|\dot\gamma_{h}\|_{L^{\infty}}$ that can be estimated in terms of $\gamma$ using \cref{thm:aPrioriBounds}. Thus, by slight abuse of notation, also for $v\in\lift\gamma\mfd$ we will take $\linInt v\in \lift{\gamma_{h}}\mfd$.
\end{remark}

\begin{remark}[Alternative interpolation operators]
	Note that the linear vector field interpolation defined by the boundary value problem~\eqref{eq:jacobi} is natural along both linear and cubic splines in the following sense:
	\begin{enumerate}
		\item Let $\gamma_h(s,t)$ be a family of linear splines, then its variation is linear in the sense $\linInt\partial_s\gamma_h=\partial_s\gamma_h$.
		Indeed, the optimality conditions in \cref{thm:optCond} imply
		$0=D_s(D_t^2\gamma_h)=D_tD_s\partial_t\gamma_h+\Rm(\partial_s\gamma_h,\partial_t\gamma_h)\partial_t\gamma_h=D_t^2\partial_s\gamma_h+\Rm(\partial_s\gamma_h,\partial_t\gamma_h)\partial_t\gamma_h$ (where we used \cref{thm:rules}), the well-known equation of Jacobi fields.
		\item For cubic splines the second derivatives $D_{t}^{2}\gamma_{h}$ are linear in the sense that $\linInt D_{t}^{2}\gamma_{h}=D_{t}^{2}\gamma_{h}$.
	\end{enumerate}
	It is obvious that there are multiple possible alternatives to generalize piecewise linear interpolation from Euclidean space to vector fields on Riemannian manifolds.
	In fact replacing~\eqref{eq:jacobi} by $D_{t}^{2}w=0$ would also work in the following convergence analysis with just slight modifications to the proofs, however, it seems a little less natural.
	Indeed, the motivation for \cref{def:linInt} is to view a (piecewise) linear vector field as a second derivative of a cubic spline (which in Euclidean space is a piecewise cubic polynomial).
  Would one consider higher order spline interpolation, that is, minimization of $\int_0^1|D_t^k\gamma_h|^2\,\d t$ under interpolation constraints,
	then again a corresponding natural definition of $\linInt$ would have been the $(2k-2)$th derivative of local minimizers of that energy
	(and the special case $k=2$ leads to \cref{def:linInt}).
\end{remark}

\subsection{Quartic convergence of cubic splines}

We aim to derive a quartic convergence rate by applying \cref{thm:perturbedSpline} to $g=\pi_{\gamma_h}\log_{\gamma_h}\!\gamma$.
To this end we need to show boundedness of $\partial_t^4g$, which can be reduced to the boundedness of $D_t^4\gamma_h$ and thus, via \cref{thm:optCond}\eqref{enm:optCond}, of $D_t^2\gamma_h$.
The derivation of such a bound occupies the major part of this paragraph, at the end of which the quartic convergence is deduced.
Exploiting $\linInt D_t^2\gamma_h=D_t^2\gamma_h$,
we could bound $D_t^2\gamma_h$ by \cref{thm:LinearInterpolation} if we knew its values at the interpolation times $t_0,\ldots,t_N$,
however, those are not available.
This is in contrast to the piecewise linear interpolation $\linInt D_t^2\gamma$ of $D_t^2\gamma$ (along $\gamma_h$), which thus can readily be bounded in $L^\infty((0,1))$.
Therefore we will show boundedness of $D_t^2\gamma_h$ by bounding its difference to $\linInt D_t^2\gamma$.
The proof, in turn, approximates this difference by $D_t^2\log_{\gamma_h}\!\gamma$,
the generalization of $\ddot\gamma_h-\ddot\gamma$ from the Euclidean case, for which there is a Galerkin orthogonality and consequently a best-approximation result.

\begin{proposition}[Approximation error for $D_t^2\gamma_h$]\label{thm:approximationErrorSecondDeriv}
For $\gamma$ four times differentiable and $h$ small enough (depending on $\mfd$ and $\gamma$) the cubic spline interpolation satisfies
\begin{align*}
\|D_t^2\log_{\gamma_h}\!\gamma-(D_t^2\gamma_h-\linInt D_t^2\gamma)\|_{L^p}
&\leq c\left(h^2 +\|d(\gamma_h,\gamma)\|_{L^p}\right)\left(\|D_t^2\gamma_h\|_{L^\infty}+1\right),\\
\|D_t^2\gamma_h-\linInt D_t^2\gamma\|_{L^2}
&\leq c\left(h^2 +\|d(\gamma_h,\gamma)\|_{L^2}\right)\left(\|D_t^2\gamma_h\|_{L^\infty}+1\right)
\end{align*}
for a constant $c$ only depending on $\mfd$ and $\gamma$.
\end{proposition}
\begin{proof}
We will estimate the difference by adding a zero and using the triangle inequality to split into easier terms.
In particular, we will add
\begin{equation*}
0=-D_t^2\log_{\gamma_h}\!\gamma+D_t(\partial_2\log_{\gamma_h}\!\gamma(\dot\gamma))+D_t(\partial_1\log_{\gamma_h}\!\gamma(\dot\gamma_h))
\end{equation*}
to obtain
\begin{align*}
D_t^2\gamma_h-\linInt D_t^2\gamma
&=-D_t^2\log_{\gamma_h}\!\gamma
+\left[D_t(\dot\gamma_h+\partial_1\log_{\gamma_h}\!\gamma(\dot\gamma_h))\right]
+\left[D_t(\partial_2\log_{\gamma_h}\!\gamma(\dot\gamma))-\linInt(\partial_2\log_{\gamma_h}\!\gamma(D_t^2\gamma))\right]\\
&=-D_t^2\log_{\gamma_h}\!\gamma
+\left[D_t^2\gamma_h+\partial_1\log_{\gamma_h}\!\gamma(D_t^2\gamma_h)\right]
+\left[\partial_1^2\log_{\gamma_h}\!\gamma(\dot\gamma_h,\dot\gamma_h)+\partial_2\partial_1\log_{\gamma_h}\!\gamma(\dot\gamma,\dot\gamma_h)\right]\\
&\quad+\left[\partial_2\log_{\gamma_h}\!\gamma(D_t^2\gamma)-\linInt(\partial_2\log_{\gamma_h}\!\gamma(D_t^2\gamma))\right]
+\left[\partial_1\partial_2\log_{\gamma_h}\!\gamma(\dot\gamma_h,\dot\gamma)+\partial_2^2\log_{\gamma_h}\!\gamma(\dot\gamma,\dot\gamma)\right],
\end{align*}
where we additionally exploited $\linInt D_t^2\gamma=\linInt(\partial_2\log_{\gamma_h}\!\gamma(D_t^2\gamma))$.
We now estimate the bracketed terms.
By \cref{thm:propertiesLinearInterpolation}, we have
\begin{align*}
	\|\partial_2\log_{\gamma_h}\!\gamma(D_t^2\gamma)-\linInt(\partial_2\log_{\gamma_h}\!\gamma(D_t^2\gamma))\|_{L^p}
	&\leq c\;h^{2}\left(\|D_{t}^{2}\partial_2\log_{\gamma_h}\!\gamma(D_t^2\gamma)\|_{L^{p}} + \|\dot\gamma\|_{L^{\infty}}^{2}\|\partial_2\log_{\gamma_h}\!\gamma(D_t^2\gamma)\|_{L^{\infty}}\right).
\end{align*}
We have
\begin{equation*}
D_t^2\partial_2\log_{\gamma_h}\!\gamma(D_t^2\gamma)
=\partial_1^2\partial_2\log_{\gamma_h}\!\gamma(D_t^2\gamma_h,D_t^2\gamma)+r,
\end{equation*}
where $r$ only contains derivatives up to third order of the logarithm evaluated in directions of derivatives of $\gamma$ up to fourth order and of $\gamma_h$ up to first order.
Thus, $r$ can be bounded by a constant (depending on $\mfd$ and $\gamma$).
Together we obtain
\begin{equation*}
\|\partial_2\log_{\gamma_h}\!\gamma(D_t^2\gamma)-\linInt(\partial_2\log_{\gamma_h}\!\gamma(D_t^2\gamma))\|_{L^p}
\leq ch^2(\|D_t^2\gamma_h\|_{L^\infty}+1).
\end{equation*}
By \cref{thm:logDerivs} we furthermore have
\begin{align*}
&\left\|\left[\partial_1^2\log_{\gamma_h}\!\gamma(\dot\gamma_h,\dot\gamma_h)+\partial_2\partial_1\log_{\gamma_h}\!\gamma(\dot\gamma,\dot\gamma_h)\right]+\left[\partial_1\partial_2\log_{\gamma_h}\!\gamma(\dot\gamma_h,\dot\gamma)+\partial_2^2\log_{\gamma_h}\!\gamma(\dot\gamma,\dot\gamma)\right]\right\|_{L^p}\\
&\hspace*{.5\textwidth}\leq c\|d(\gamma_h,\gamma)\|_{L^p}(\|\dot\gamma\|_{L^\infty}+\|\dot\gamma_h\|_{L^\infty})^2
\leq\tilde c\|d(\gamma_h,\gamma)\|_{L^p},\\
&\left\|D_t^2\gamma_h+\partial_1\log_{\gamma_h}\!\gamma(D_t^2\gamma_h)\right\|_{L^p}
\leq c\|d(\gamma_h,\gamma)\|_{L^\infty}\|d(\gamma_h,\gamma)\|_{L^p}\|D_t^2\gamma_h\|_{L^\infty}
\leq\tilde c\|d(\gamma_h,\gamma)\|_{L^p}\|D_t^2\gamma_h\|_{L^\infty}
\end{align*}
so that in summary
\begin{equation*}
\|D_t^2\log_{\gamma_h}\!\gamma-(D_t^2\gamma_h-\linInt D_t^2\gamma)\|_{L^p}
\leq c\left(h^2 +\|d(\gamma_h,\gamma)\|_{L^p}\right)\left(\|D_t^2\gamma_h\|_{L^\infty}+1\right).
\end{equation*}
Next let $w\in\lift{\gamma_h}\mfd$ be piecewise linear in the sense of \cref{def:linInt}, then applying integration by parts twice we obtain
\begin{equation*}
	\int_{0}^{1}(-D_t^2\log_{\gamma_h}\!\gamma,w)\,\d t
	=\sum_{k=1}^{N}
	\int_{t_{k-1}}^{t_{k}}(\log_{\gamma_h}\!\gamma,\Rm(w,\dot\gamma_h)\dot\gamma_h)\,\d t
	\leq c\|d(\gamma_{h},\gamma)\|_{L^2}\|w\|_{L^2}\|\dot\gamma_h\|_{L^\infty}^2
	\leq\tilde c\|d(\gamma_h,\gamma)\|_{L^2}\|w\|_{L^2}.
\end{equation*}
By the above we have
\begin{equation*}
	\int_{0}^{1}(D_t^2\gamma_h-\linInt D_t^2\gamma,w)\,\d t
	\leq c\|w\|_{L^2}
	\left(h^2 +\|d(\gamma_h,\gamma)\|_{L^2}\right)\left(\|D_t^2\gamma_h\|_{L^\infty}+1\right).
\end{equation*}
The choice $w=D_t^2\gamma_h-\linInt D_t^2\gamma$ now finishes the proof.
\end{proof}

\begin{corollary}[Uniform boundedness of $D_t^2\gamma_h$]\label{thm:BoundSecondDeriv}
For $\gamma$ four times differentiable and $h$ small enough (depending on $\mfd$ and $\gamma$) the cubic spline interpolation satisfies
\begin{equation*}
\|D_t^2\gamma_h\|_{L^\infty}
\leq3\|D_t^2\gamma\|_{L^\infty} + ch^{\frac32}
\qquad\text{and}\qquad
\|D_t^2\log_{\gamma_h}\!\gamma\|_{L^2},
\|D_t^2\gamma_h-\linInt D_t^2\gamma\|_{L^2}
\leq ch^2
\end{equation*}
for a constant $c$ only depending on $\mfd$ and $\gamma$.
\end{corollary}
\begin{proof}
Using \cref{thm:LinearInterpolation} we have
\begin{equation*}
\|D_t^2\gamma_h\|_{L^\infty}
\leq\|D_t^2\gamma_h-\linInt D_t^2\gamma\|_{L^\infty}+\|\linInt D_t^2\gamma\|_{L^\infty}
\leq\|D_t^2\gamma_h-\linInt D_t^2\gamma\|_{L^\infty}+2\|D_t^2\gamma\|_{L^\infty}.
\end{equation*}
By \cref{thm:propertiesLinearInterpolation} and $D_t^2\gamma_h=\linInt D_t^2\gamma_h$ we have
\begin{align*}
	\|D_t^2\gamma_h-\linInt D_t^2\gamma\|_{L^{\infty}}&\leq c\left(1+h^{-\frac{1}{2}}\right)\|D_t^2\gamma_h-\linInt D_t^2\gamma\|_{L^{2}}.
\end{align*}
Using \cref{thm:approximationErrorSecondDeriv} and \cref{thm:quadraticConvergenceI} this turns into 
\begin{align*}
	\|D_t^2\gamma_h-\linInt D_t^2\gamma\|_{L^{\infty}}&\leq c\left(h^2+h^{\frac{3}{2}}\right)(\|D_t^2\gamma_h\|_{L^\infty}+1)\;,
\end{align*}
and thus by \cref{thm:LinearInterpolation}
\begin{equation*}
\|D_t^2\gamma_h\|_{L^\infty}
\leq2ch^{\frac32}(1+\|D_t^2\gamma_h\|_{L^\infty})+2\|D_t^2\gamma\|_{L^\infty}.
\end{equation*}
Taking $h$ small enough such that $ch^{\frac32}\leq\frac16$, we arrive at the desired bound for $\|D_t^2\gamma_h\|_{L^\infty}$.
This bound in turn can be applied in \cref{thm:approximationErrorSecondDeriv} together with \cref{thm:quadraticConvergenceI} to yield the desired bound on $\|D_t^2\log_{\gamma_h}\!\gamma\|_{L^2}$ and $\|D_t^2\gamma_h-\linInt D_t^2\gamma\|_{L^2}$.
\end{proof}

\begin{remark}[Boundedness of $\dot\gamma_h$ for linear splines]
Note that a slightly weaker boundedness result than \cref{thm:boundFirstDeriv} for linear splines could have been obtained
following exactly the same argument as the one above for cubic splines.
In that case, one would have used a piecewise constant interpolation operator $\constInt$ on vector fields $v\in\lift\gamma\mfd$
which would simply parallel transport $v(t_{k-1})$ along $\gamma((t_{k-1},t_k))$
and which would satisfy analogous properties to those of $\linInt$ in \cref{thm:propertiesLinearInterpolation}.
In the subsequent estimate of
\begin{equation*}
\dot\gamma_h-\constInt\dot\gamma
=-D_t\log_{\gamma_h}\!\gamma
+\left[\dot\gamma_h+\partial_1\log_{\gamma_h}\!\gamma(\dot\gamma_h)\right]
+\left[\partial_2\log_{\gamma_h}\!\gamma(\dot\gamma)-\constInt(\partial_2\log_{\gamma_h}\!\gamma(\dot\gamma))\right]
\end{equation*}
one would this time exploit the orthogonality property
$\int_0^1(D_t\log_{\gamma_h}\!\gamma,w)\,\d t=0$ for any piecewise constant vector field along $\gamma_h$.
\end{remark}

With the above preparations we can now turn to the final convergence result.
\Cref{thm:BoundSecondDeriv} and \cref{thm:quadraticConvergence} already prove the desired result $\|d(\gamma_h,\gamma)\|_{L^2}\leq ch^4$ in the $L^2$-norm;
for the $L^\infty$-norm (or any other $L^p$-norm) we will proceed by bounding $\|D_t^4\log_{\gamma_h}\!\gamma\|_{L^\infty}$.

\begin{corollary}[Quartic convergence]
For $\gamma$ four times differentiable, $p\in[1,\infty]$, and $h$ small enough (depending on $\mfd$ and $\gamma$) the cubic spline interpolation satisfies
\begin{equation*}
\|D_t^2\log_{\gamma_h}\!\gamma\|_{L^p}\leq ch^2,
\qquad
\|d(\gamma_h,\gamma)\|_{L^p}
\leq ch^4,
\end{equation*}
where the constant $c>0$ only depends on $\mfd$, $\gamma$, and $p$.
\end{corollary}
\begin{proof}
We just show the $L^\infty$ estimate, which is the strongest among all $p$,
and apply \cref{thm:perturbedSpline} to $g=\pi_{\gamma_h}\log_{\gamma_h}\!\gamma$ to obtain the desired result.
To this end it remains to show uniform boundedness of $|\partial_t^4g|=|D_t^4\log_{\gamma_h}\!\gamma|$ on each interval $I_k=(t_{k-1},t_k)$. Now
\begin{equation*}
D_t^4\log_{\gamma_h}\!\gamma
=\partial_1\log_{\gamma_h}\!\gamma(D_t^4\gamma_h)
+c_1\partial_1^2\log_{\gamma_h}\!\gamma(D_t^3\gamma_h,D_t\gamma)
+c_2\partial_2\partial_1\log_{\gamma_h}\!\gamma(D_t^3\gamma_h,D_t\gamma_h)
+r
\end{equation*}
with some integers $c_1,c_2$, where $r$ is a term containing derivatives of the logarithm up to fourth order
evaluated in directions of derivatives of $\gamma$ up to fourth order and of $\gamma_h$ up to second order.
Thus, $|r|$ can be bounded by a constant (depending on $\mfd$ and $\gamma$).
Consequently, exploiting \cref{thm:logDerivs} it remains to show uniform boundedness of $|D_t^4\gamma_h|$ and $d(\gamma_h,\gamma)|D_t^3\gamma_h|$,
where by \cref{thm:quadraticConvergenceII} the latter is no larger than $h^2|D_t^3\gamma_h|$ up to a constant factor
(in fact, below we will even obtain boundedness of $h|D_t^3\gamma_h|$).
By \cref{thm:optCond}\eqref{enm:optCond} as well as \cref{thm:aPrioriBounds}\eqref{enm:aPrioriBoundsCubic} and \cref{thm:BoundSecondDeriv} we have
\begin{equation*}
|D_t^4\gamma_h|=|\Rm(D_t^2\gamma_h,\dot\gamma_h)\dot\gamma_h|\leq C
\end{equation*}
for a constant $C>0$ depending on $\mfd$ and $\gamma$.
Using the scaled Gagliardo--Nirenberg--Sobolev inequality or an inverse inequality, we then obtain
\begin{equation*}
\|D_t^3\gamma_h\|_{L^\infty(I_k)}
\leq c\|D_t^4\gamma_h\|_{L^\infty(I_k)}^{1/2}\|D_t^2\gamma_h\|_{L^\infty(I_k)}^{1/2}
+ch^{-1}\|D_t^2\gamma_h\|_{L^\infty(I_k)}
\leq ch^{-1},
\end{equation*}
which finally proves boundedness $|\partial_t^4g|<C$ for a constant $C$ only depending on $\mfd$ and $\gamma$.

The estimate of $\|D_t^2\log_{\gamma_h}\!\gamma\|_{L^\infty}$ now follows again from the the scaled Gagliardo--Nirenberg--Sobolev inequality.
Indeed, there exists $c>0$ such that on each interval $I_k=(t_{k-1},t_k)$
\begin{equation*}
\|D_t^2\log_{\gamma_h}\!\gamma\|_{L^\infty(I_k)}
\leq c\|D_t^4\log_{\gamma_h}\!\gamma\|_{L^\infty(I_k)}^{1/2}\|\log_{\gamma_h}\!\gamma\|_{L^\infty(I_k)}^{1/2}+ch^{-2}\|\log_{\gamma_h}\gamma\|_{L^\infty(I_k)}.
\end{equation*}
Using $\|d(\gamma_h,\gamma)\|_{L^\infty}\leq Ch^4$ as well as the uniform boundedness of $D_t^4\log_{\gamma_h}\!\gamma$, the estimate turns into
\begin{equation*}
\|D_t^2\log_{\gamma_h}\!\gamma\|_{L^\infty(I_k)}
\leq c(\sqrt C\|D_t^4\log_{\gamma_h}\!\gamma\|_{L^\infty(I_k)}^{1/2}+C)h^{2}.
\qedhere
\end{equation*}
\end{proof}

\section*{Acknowledgements}
B.\,Wirth's research was supported by the Alfried Krupp Prize for Young University Teachers awarded by the Alfried Krupp von Bohlen und Halbach-Stiftung.
He also acknowledges support by the Deutsche Forschungsgemeinschaft (DFG, German Research Foundation)
under Germany's Excellence Strategy EXC 2044--390685587, Mathematics Münster: Dynamics--Geometry--Structure,
and under the Collaborative Research Centre 1450, InSight, University of M\"unster.

\bibliographystyle{plain}
\bibliography{notes}

\end{document}